\documentclass[review]{elsarticle}
\usepackage{lineno,hyperref}
\usepackage{algorithm}
\usepackage[normalem]{ulem}
\modulolinenumbers[5]

\usepackage{graphicx,amsmath,amssymb,times,cleveref,subfigure,url,comment,epsfig,color,mathrsfs,mathtools}

\newtheorem{theorem}{Theorem}

\newtheorem{assumption}[theorem]{Assumption}

\newtheorem{lemma}[theorem]{Lemma}

\newtheorem{proposition}[theorem]{Proposition}

\newtheorem{remark}[theorem]{Remark}

\newtheorem{definition}[theorem]{Definition}

\newenvironment{proof}[1][Proof]{\textbf{#1.} }{\ \hspace*{\fill} \rule{0.5em}{0.5em}}

\graphicspath{{Figures/}}

\usepackage{ifthen}
\newboolean{showcomments}
\setboolean{showcomments}{true}

\newcommand{\note}[1]{\ifthenelse{\boolean{showcomments}} {\textcolor{red}{#1}}{}}

\usepackage{soul,xcolor}
\begin{document}
\begin{frontmatter}
\title{Data-Driven Operator Theoretic Methods for Phase Space Learning and Analysis}

\author{Sai Pushpak Nandanoori*}\cortext[mycorrespondingauthor]{Corresponding author}
\ead{saipushpak.n@pnnl.gov} 

\author{Subhrajit Sinha} 
\ead{subhrajit.sinha@pnnl.gov}

\author{Enoch Yeung}
\ead{eyeung@ucsb.edu}

\fntext[fn1]{Sai Pushpak, Subhrajit Sinha are with the Pacific Northwest National Laboratory, Richland, WA 99354}
 
\fntext[fn2]{Enoch Yeung is with the Department of Mechanical Engineering, University of California Santa Barbara, CA 93106}





\begin{abstract}

This paper uses data-driven operator theoretic approaches to explore the global phase space of a dynamical system. We defined conditions for discovering new invariant subspaces in the state space of a dynamical system starting from an invariant subspace based on the spectral properties of the Koopman operator. When the system evolution is known locally in several invariant subspaces in the state space of a dynamical system, a phase space stitching result is derived that yields the global Koopman operator. Additionally, in the case of equivariant systems, a phase space stitching result is developed to identify the global Koopman operator using the symmetry properties between the invariant subspaces of the dynamical system and time-series data from any one of the invariant subspaces. 
Finally, these results are extended to topologically conjugate dynamical systems; in particular, the relation between the Koopman tuple of topologically conjugate systems is established. The proposed results are demonstrated on several second-order nonlinear dynamical systems including a bistable toggle switch. 
Our method elucidates a strategy for designing discovery experiments: experiment execution can be done in many steps, and models from different invariant subspaces can be combined to approximate the global Koopman operator. 
\end{abstract}


\end{frontmatter}


\section{Introduction}


Multiple equilibria, periodic orbits, limit cycles, chaotic attractors, and other complicated behaviors are common in nonlinear dynamical systems, and identifying such systems from data can be difficult at times. Furthermore, most known approaches for studying or analyzing nonlinear systems theory necessitate knowledge of mathematical models governing the system of interest, and most methods necessitate the computing of a unique energy function for the underlying system. It is difficult to compute such an energy function, especially for a higher-dimensional system. Recent research suggests that nonlinear systems can be examined using data alone, without the use of any mathematical models.


The data from a nonlinear system is mapped to a higher-dimensional observable space, where it may be examined successfully utilizing transfer operators such as the Perron-Frobenius and Koopman operators \cite{mezic2005spectral, lasota2013chaos,rowley2009spectral,vaidya_lyapunov_measure,mezic2013analysis, yeung2017learning}. Because both operators are adjoint, a dynamical system can theoretically be analyzed using any of them. However, in terms of application to real-world circumstances, each of the operators has its own set of advantages and disadvantages. In comparison to the Perron-Frobenius operator, Koopman operators are more commonly utilized for data-driven analysis.

Bernard Koopman introduced the Koopman operator in \cite{koopman1931hamiltonian} and this seminal work became a popular tool with \cite{mezic2005spectral} and several other works followed where Koopman operator theory is used in the system identification \cite{rowley2009spectral,schmid2010dynamic}, for control design \cite{proctor2016dynamic,brunton2016koopman,huang2018feedback}, sensor fusion \cite{williams2015data} and analysis of spectrally conjugate systems \cite{mezic2020spectrum},  finding observability gramians or observers \cite{vaidya2007observability,surana2016linear,yeung2018koopman}, study of chaotic systems \cite{arbabi2017ergodic,arbabi2017study}, data-driven information transfer \cite{sinha_IT_CDC_2015,sinha_IT_CDC_2016,sinha_IT_ICC_2017}, data-driven based causal inference in dynamical systems \cite{sinha_data_IT_journal}, in data-driven classification of power system stability \cite{sinha_data_IT_stability}, and in power system coherency identification, power system stability monitoring  \cite{susuki2016applied,barocio2014dynamic,hernandez2018nonlinear,barocio2014dynamic,raak2015data}, dynamic state estimation in power networks \cite{netto2018robust}, fault, cyber-attack localization in cyber-physical systems \cite{ramos2019dynamic,nandanoori2020model}, computational neuroscience applications \cite{marrouch2019data} and data assimilation for climate forecasting \cite{slawinska2019quantum}. 

The Koopman operator definition is an infinite-dimensional linear operator that yields the evolution of a dynamical system on an infinite-dimensional space of functions. Finding an infinite dimensional (linear) operator is computationally infeasible, hence many approaches are devised to best approximate the Koopman operator in finite dimensional space \cite{rowley2009spectral,tu2013dynamic,williams2014kernel,williams2015data,sinha_robust_acc,sinha_robust_journal,sinha_sparse_acc,boddupalli2019koopman, sinha2020data,sinha2020equivariant,bakker2019learning}. Most popular methods include dynamic mode decomposition (DMD) \cite{rowley2009spectral,tu2013dynamic}; extended dynamic mode decomposition (E-DMD) \cite{williams2015data,li2017extended}; kernel dynamic mode decomposition (K-DMD) \cite{williams2014kernel}, naturally structured dynamic mode decomposition (NS-DMD) \cite{huang2017data}, Hankel-DMD \cite{arbabi2017ergodic}, deep dynamic mode decomposition (deep-DMD)\cite{yeung2017learning,li2017extended,takeishi2017learning,lusch2018deep,hasnain2020steady}. All these methods are data-driven and accuracy of some such approximations are discussed in \cite{zhang2017evaluating}. The authors of \cite{johnson2018class} look at the space of Koopman observables, which includes the underlying system's states. The authors of \cite{johnson2018class} also demonstrate the fidelity of this class of observable functions and discuss performance tuning options.


\subsection{Problem Statement}
The major goal of this research is to use data to examine the global phase space of nonlinear dynamical systems. We identified a few natural questions that occur during experimental phase space analysis when there is less knowledge about the underlying system and the experiments are costly. Given the time-series data corresponding to an invariant subspace of a nonlinear system, it is of interest to know as additional time-series data emerges, can the other invariant subspaces be discovered? Is it possible to uncover the global phase space? On the contrary, suppose if the evolution of a system in different subspaces is known, can they be combined to analyze the global phase space? Is it possible to examine the global phase space with only one invariant subspace and knowledge of the nature of underlying symmetry? Finally, can the phase space properties of a topologically conjugate system be determined if the phase space properties of a system are known?

This work attempts to address these questions by leveraging the tools from Koopman operator theory. The experimental phase space discovery procedure of a classic biological nonlinear network, the bistable toggle switch, inspired this work.

\subsection{Summary of Contributions}
The Koopman operator and its spectral properties are utilized to investigate the global phase space of nonlinear dynamical systems in this work. Our contributions are as follows:
\begin{itemize}
    \item {We provided conditions for updating the local Koopman operator when new spatial time-series data emerges.} 
    \item We developed conditions for discovering invariant subspaces from a global phase space, and vice versa.
    \item We presented a phase space stitching result in which many local Koopman operators were fused to produce a (single) global Koopman operator.
    \item We extended our phase space stitching result to obtain the global Koopman operator for equivariant dynamical systems when data is available only in any one of the invariant subspaces but the knowledge about the underlying symmetry between the invariant subspaces is known.
    \item Results were developed in order to comprehend the global phase space of topologically conjugate systems.
\end{itemize}

We hope the phase space learning and analysis from this work serves as a basis for designing and planning future experiments. This work is an extension to our previous works \cite{sinha2020equivariant,nandanoori2020data}. {In particular, the partitioning of the state space into multiple invariant sets, and the phase space stitching result are introduced in} \cite{nandanoori2020data}. The Koopman operator theory for equivariant systems is developed in \cite{sinha2020equivariant}. The rest of the contributions are exclusive to this work.  


The rest of the paper is organized as follows. Section \ref{sec:DMD_variants} describes the preliminaries for operator theoretic methods and consists of a brief overview on the computation of finite-dimensional approximation of Koopman operators. Some thoughts on the relation between the state space and the observable space are discussed in Section \ref{sec:discussion}. Discovery of new invariant subspaces from the spectral properties of the Koopman operator and the phase space stitching results along with a use case for equivariant dynamical systems is presented in Section \ref{sec:global_phase_space}. Using Koopman operators, global phase space learning and analysis for topologically conjugate dynamical systems is presented in Section \ref{sec:TC_systems}. The proposed Koopman operator theoretic methods are illustrated on  second-order dynamical systems in Section \ref{sec:simulation} and the paper concludes with final remarks in Section \ref{sec:conclusion}. 


\section{Mathematical Preliminaries and Dynamic Mode Decomposition (DMD) Variants}
\label{sec:DMD_variants}
We recall the necessary concepts that forms a base for the results in this work from \cite{mezic2005spectral,lasota2013chaos,mezic2013analysis,williams2015data,budivsic2012applied}. Consider the following discrete-time nonlinear system
\begin{align}
    x_{t+1} = T(x_t)
    \label{eq:DT_NL_sys}
\end{align}
 where $x \in {\cal M} \subseteq  \mathbb{R}^n$ and $T:{\cal M} \to {\cal M}$. The phase space ${\cal M}$ is assumed to be a compact  manifold with Borel $\sigma$ algebra $\mathcal{B({\cal M})}$,  $\mu$ denotes the measure \cite{mezic2005spectral,lasota2013chaos} and the nonlinear map $T$ is assumed to be $\mathcal{C}^1$ continuous.

Let $\mathbb{N}$ denote the set of natural numbers. Recall the definition of invariant set from the classical dynamical systems theory as follows. 
\begin{definition}[Invariant Subspace]
Let $M \subseteq {\cal M}$ be a subspace of the dynamical system, Eq. \eqref{eq:DT_NL_sys}. Then, the subspace $M$ is said to be (positively) invariant if for every $x_0 \in M$, $T^n(x_0) \in M$ for all $n\in\mathbb{N}$.
\end{definition}
Define an observable function on the state space to be a scalar valued function ${\psi}:\mathcal{M} \rightarrow \mathbb{C}$, where $\mathbb{C}$ denotes the set of complex numbers. The scalar valued function ${\psi}$ belongs to the space of functions ${\cal G}$ acting on elements of $\cal M$. Usually for any dynamical system, the output function describes the evolution of the states and this output function can be considered as an observable function. Moreover, any differentiable function of the states can be chosen as an observable function. Normally, square integrable functions defined on ${\cal M}$ forms a good choice of observable functions. Hence, the space of functions, ${\cal G} = L_2({\cal M}, {\cal B}, \mu)$ is infinite dimensional. The Koopman operator can now be defined as follows. 
\begin{definition}[Koopman operator] For a dynamical system $x\mapsto T(x)$ and for ${\psi} \in \mathcal{G}$, the Koopman operator $({\mathbb{U}} : {\cal G} \to {\cal G})$ associated with the dynamical system is defined as
\[ \left[\mathbb{U} {\psi}\right] (x) = \psi(T(x)). \]
\end{definition}

The Koopman operator $\mathbb{U}$ is defined on the space of functions and is a linear operator. The function space ${\cal G}$ is invariant under the action of the Koopman operator. Essentially, instead of studying the time evolution of nonlinear system \eqref{eq:DT_NL_sys} on the state space ${\cal M}$, we study the evolution of \textit{observables} in an infinite dimensional space where the evolution is linear. Apart from being linear, the Koopman operator is also a positive operator: for any $g \geq 0$, $[\mathbb{U} g] (x) \geq 0$.

The infinite dimensional Koopman operator admits discrete and continuous spectrum. In the scope of this work, we are interested in the discrete spectrum of the Koopman operator. The following section describes the popular methods to identify a finite dimensional approximate of the Koopman operator from the time-series data. 

\subsection{DMD Variants}
This subsection discusses the computation of the approximate Koopman operator from the time-series data using prominent methods such as dynamic mode decomposition (DMD), extended dynamic mode decomposition (EDMD) and deep dynamic mode decomposition (deepDMD). 

Schmid and others introduced DMD in \cite{schmid2010dynamic} to approximate the Koopman operator that describes the coherent features of fluid flow and extract dynamic information from flow fields through the spectrum of the approximate Koopman operator. The authors in \cite{williams2015data} generalized the idea of DMD to introduce EDMD.  In EDMD, an extended dictionary of functions is used to approximate the action of the infinite dimensional Koopman operator.

Consider the time-series data $X = [x_0,\;x_1,\cdots , x_k]$ from an experiment or simulation of a dynamical system and stack them as shown below. 
\begin{align*}
X_{pr} = [x_0, x_1, \dots, x_{k-1}], \qquad X_{fw} = [x_1,x_2, \dots,x_k],
\end{align*}
where for every $i$, $x_i \in {\cal M}$.  Assuming knowledge of the space ${\cal G} = L_2({\cal M},{\cal B}, \mu)$ we can define the set of dictionary functions ${\cal D} = \{\psi_1,\dots,\psi_N\}$ where $\psi_i \in L_2({\cal M},{\cal B},\mu)$ and $\psi_i:{\cal M} \to \mathbb{C}$ are scalar valued functions. {Each observable function in the dictionary ${\cal D}$ are assumed to be bounded.} Let the span of these $N$ dictionary functions in ${\cal D}$ be denoted by ${\cal G}_{\cal D}$ such that ${\cal G}_{\cal D} \subset {\cal G}$. 
%

Since we are considering either time-series data from an experiment or simulation in the scope of this work, it is natural to describe their evolution as a discrete-time system, and so we describe the Koopman operator theory in  discrete-time setting. All of the results presented in this paper, however, are valid in the continuous-time setting as well. 

Define a vector valued observable function ${\Psi}:{\cal M} \to \mathbb{C}^N$ such that,
\begin{align*}
{\Psi}(x):=\begin{bmatrix}\psi_1(x) & \psi_2(x) & \cdots & \psi_N(x)\end{bmatrix}.
\end{align*}
Then by definition for a function $\hat{\psi}_i \in {\cal G}_{\cal D}$,  $\hat{\psi}_i$ can be expressed as an inner product of a coefficient vector ${a}_i$ with the dictionary functions as shown below. 
\begin{align*}
\hat{\psi}_i(x) = \sum_{\ell=1}^N a_{i \ell} \; \psi_{ \ell} =  {\Psi}(x) {a}_i
\end{align*}
for every $x \in X$. Similarly for $\hat{{\Psi}}(x) = \begin{bmatrix}\hat \psi_1(x) & \hat \psi_2(x) & \cdots & \hat \psi_N(x)\end{bmatrix}$, we obtain, 
\begin{align*}
  \hat{{ \Psi}}( x) =  \Psi(x) A
\end{align*}
where $A = \begin{bmatrix}  a_1 & a_2 & \dots & a_N \end{bmatrix} \in \mathbb{R}^{N \times N}$.  

For every snapshot pair, $x_{pr}, x_{fw}$ such that $x_{fw} = T(x_{pr})$, we obtain, 
\begin{align*}
    \hat{\Psi}(x_{fw}) & = \left[ \mathbb{U} \hat{\Psi}\right](x_{pr}) \\
    & =  \Psi  (x_{pr}) \hat{\mathbb{U}} A + r_X(x)
\end{align*}
where $\hat{\mathbb{U}}$ is a projection of $\mathbb{U}$ on the snapshot space. The residual, $r_X(x) \in \mathbb{C}^{\mathbb{N}}$ is uniformly bounded on the snapshot space $X$ and $r_X(\cdot) \in \mathcal{G} \setminus \mathcal{G}_{\cal D}$. 
%
%
%
Therefore the Koopman learning problem now results in minimizing the residual, $r_X$ which leads to the following optimization problem. 
\begin{align}
    \min_{A, \hat{\mathbb{U}}} \;\; \parallel  \left(\Psi(x_{fw}) -  \Psi(x_{pr})\right) \hat{\mathbb{U}} A \parallel    
    \label{eq:EDMD_int_step}
\end{align}
The above optimization problem in \eqref{eq:EDMD_int_step} is equivalent to $$\min_{\hat{\mathbb{U}}} \;\; \parallel \Psi(x_{fw}) - \Psi(x_{pr}) \hat{\mathbb{U}} \parallel. $$
Minimizing the residual over the given snapshot space, $X_{pr}$ and $X_{fw}$ leads to the following minimization problem. 
\begin{align*}
    \min_{\tilde{\mathbb{U}}} \;\; \left|\left|  \left(\begin{bmatrix} {\Psi}(x_{1})\\ \vdots \\ {\Psi}(x_{k}) \end{bmatrix} - \begin{bmatrix}{\Psi}(x_{0}) \\ \vdots \\  {\Psi}(x_{k-1})\end{bmatrix} \tilde{{\mathbb{U}}} \right)  \right|\right|  
\end{align*}
where \[{\tilde{\mathbb{U}}} = \begin{bmatrix} {\cal K}_1 & 0 & \hdots & 0 \\ 0 & {\cal K}_2 & \hdots & 0 \\ \vdots & \vdots & \ddots & \vdots \\0 &0 &\hdots & {\cal K}_k\end{bmatrix}\in\mathbb{C}^{kN \times kN }
\]

Note that we could thus parameterize $\tilde{\mathbb{U}}$ in terms of each snapshot, thus mapping each snapshot to the diagonal entries ${\cal K}_j$ for $j = 1,...,k.$  In practice, with standard DMD algorithms, it is common to assume that the action of $\tilde{\mathbb{U}}$ is homogeneous across snapshot space and that all ${\cal K}_j = {\cal K} \in \mathbb{C}^{N \times N}.$  The precise implications of this assumption are not immediately clear, though it is a widely employed practice within the DMD community.  This results in the following classical and frequently formulated EDMD problem: 

\begin{equation}
\min_{\cal K}\parallel {Y_{fw}} - {Y_{pr}} {\cal K} \parallel_2^2,
\label{edmd_op}
\end{equation}
where
\begin{align*}
& {Y_{pr}} = \begin{bmatrix} {\Psi}(x_0) \\ {\Psi}(x_1) \\ \vdots \\ {\Psi}(x_{k-1}) \end{bmatrix}, \qquad {Y_{fw}} = \begin{bmatrix} {\Psi}(x_1) \\ {\Psi}(x_2) \\ \vdots \\ {\Psi}(x_k) \end{bmatrix}
\end{align*}
and ${\cal K}$ is the finite dimensional approximation of the Koopman operator $\mathbb{U}$. For convenience, we minimize the upper bound of the induced matrix 2-norm, which is the Frobenius norm
\begin{equation}
\min_{\cal K}\parallel {Y_{pr}}{\cal K}-{Y_{fw}}\parallel_F^2,
\label{edmd_op}
\end{equation}
with ${\cal K}\in\mathbb{C}^{N\times N}$. The solution to the optimization problem \eqref{edmd_op} can be obtained explicitly and is given by 
\begin{align*}
{\cal K}={Y_{pr}}^\dagger {Y_{fw}}
\end{align*}
where ${Y_{pr}}^{\dagger}$ is the Moore-Penrose inverse of $Y_{pr}$. Note that DMD is a special case of the EDMD algorithm with ${\Psi}(x) = x$. Hereafter, the infinite dimensional Koopman is denoted by $\mathbb{U}$ and the finite dimensional approximation of the Koopman operator is denoted by ${\cal K}$.

EDMD is formulated under the assumption that the observable functions for the underlying system are available. As a result, using EDMD to compute the approximation Koopman operator simply comes down to solving a least squares problem (similar to DMD however with a richer observable space that could better capture the properties of the underlying nonlinear system such as an attractor set).  
It's natural to wonder if considering the state inclusive observable space, that is, integrating the observable functions defined in EDMD with observables (measured functions of states) considered in DMD, has any benefit. The authors in \cite{johnson2018class} show the state inclusive observables defined on certain classes of dictionaries yield approximately Koopman invariant subspaces and the corresponding Koopman can simultaneously learn the underlying system dynamics and its corresponding flow.
%


Identifying the observable functions for a large scale system (perhaps with thousands or millions of states) is difficult. Furthermore, computing the Moore-Penrose inverse corresponding to this large-scale system for data acquired under a variety of initial conditions adds a new level to the challenge. DeepDMD approaches, which were recently developed in \cite{yeung2017learning,li2017extended,takeishi2017learning}, provide an alternative in this regard, as they make no assumptions on the choice of observable functions. Furthermore, because deepDMD solves for minimizing Eq. \eqref{edmd_op} as an optimization problem with decision variables as observable functions and the approximate Koopman operator, it does not require an explicit computation of Moore-Penrose inverse. For further implementation details and detailed discussion on deepDMD, we refer the readers to the works  \cite{yeung2017learning,li2017extended,takeishi2017learning}. 

Corresponding to the time-series data $X_{pr}, X_{fw}$, if the dictionary functions are chosen in such a way that the optimal cost in Eq. \eqref{edmd_op}  is in the acceptable range of accuracy, then the spectral properties of the finite dimensional Koopman operator and in particular, the dominant eigenvalues and their respective eigenfunctions unfold the nonlinear system properties.

\section{Discussion on the Observable Functions}
\label{sec:discussion}
Consider the dynamical system given in \eqref{eq:DT_NL_sys}. In general, it can have multiple attractors for example, in the case of a bistable toggle switch which is a second-order system (as shown in Eq. \eqref{eq:bistable_toggle_switch} in Section \ref{sec:simulation}). 
Assume that the coordinate space defined by these state inclusive dictionary functions is topologically conjugate to a linear system when the dictionary functions are evaluated alongside the states. As a result, the first two elements of the Koopman equation correspond to the original dynamical system in the case of the bistable toggle switch. While the first two elements produce a multistable phase portrait, the appearance is linear when raised into higher dimensional space.


In the following we briefly recall the relation between the observable functions and the Koopman eigenfunctions and introduce the Koopman modes. 

\subsection{Relation Between the Observables and the Koopman Eigenfunctions}
Let $\lambda_j$ and $\phi_j$ denote the eigenvalues and eigenfunctions corresponding to the point spectrum of the infinite dimensional Koopman operator such that 
%
\[\mathbb{U} \phi_j =  \lambda_j \phi_j, \quad j = 1,2,\dots  \]
Let ${f}$ be a vector-valued observable such that 
\[
{f} = \begin{bmatrix} f_1 \\ f_2 \\ \vdots \\ f_m
\end{bmatrix}
\]
where each $f_i \in {\cal G}$. Suppose if each $f_i$ lies in the span of eigenfunctions of the Koopman operator, then ${f}$ can be expressed as a linear combination of the eigenfunctions as follows. 
%
%
\begin{align}
    {f}(x) = \sum_{i=1}^{\infty} \phi_i(x) {\vartheta}_i = \sum_{i=1}^{\infty} \lambda_i^j \phi_i(x_0) {\vartheta}_i.
    \label{eq:KMD_expansion}
\end{align}
%
The vector coefficients, ${\vartheta}_i$ in Eq. \eqref{eq:KMD_expansion} are the Koopman modes and this expansion is usually referred to as the Koopman mode decomposition (KMD). The magnitude and frequency of each Koopman mode is given by the magnitude and phase of its corresponding eigenvalue. The spatio-temporal aspects in the dynamics are encoded in the eigenvalues (temporal signatures) and the term $\phi_i(x_0)  \vartheta_i$ capture the spatial signatures. Notice that the Koopman modes are indeed a function of the initial condition. There are several methods to compute these Koopman modes and interested readers can refer to \cite{susuki2016applied,budivsic2012applied,susuki2015prony,bagheri2013koopman,arbabi2017study,sharma2016correspondence}. The Koopman eigenvalues, Koopman eigenfunctions and the Koopman modes are usually referred to as the Koopman tuple. 

Recall that the learnt linear dynamics applying Koopman operator theory evolves on the space of observables. The next section describes the relation between the space of observables and the state space. 

\subsection{Relation between the Observable Space and the State Space}
For the nonlinear dynamical system, Eq. \eqref{eq:DT_NL_sys}, let $\mathbb{U}$ be the corresponding Koopman operator such that  
%
\[
\mathbb{U} \Psi(x_t) = {\Psi}(x_{t+1}).
\]

If $\phi(x)$ is an eigenfunction of $\mathbb{U}$ with eigenvalue $\lambda$, then 
\[
\mathbb{U}\phi(x_t) = \lambda\phi(x_t) = \phi(x_{t+1}).
\]

Among the eigenfunctions of the Koopman operator $\mathbb{U}$, the ones with unit eigenvalue are of special interest as they correspond to the attractor sets (like equilibrium points, limit cycles) of the underlying system \cite{budivsic2012applied}. In particular, the number of unit eigenvalues of the Koopman operator gives the number of attractor sets for the dynamical system. 
%
Hence, to find all the attractor sets of an invariant subspace ${\cal M}$, it suffices to solve 
\begin{equation*}
    \phi(x) = \mathbb{U}\phi(x)
\end{equation*}
Moreover, note that for an equilibrium point $x\in\mathcal{M}$, $\phi(x)={\mathbb{U}^n}\phi(x)$ for all $n\in\mathbb{N}$, while for any point on a limit cycle there exists some $n_0\in\mathbb{N}$ such that 
\begin{equation*}
    \phi(x_*) = {\mathbb{U}^{kn_0}}\phi(x_*), \quad \forall \quad  k\in\mathbb{N},
\end{equation*}
such that corresponding to the eigenvalue $\lambda = 1$, the eigenfunction, $\phi$ satisfies
\begin{equation*}
    \phi(T^{n_0}(x_*)) = \mathbb{U}^{n_0} \phi(x_*) = \lambda^{n_0} \phi(x_*)    
\end{equation*}
Next, we briefly discuss the relation between the evolution of the original nonlinear system on the state space and the observable space. Consider the discrete-time nonlinear system as shown in Eq.  \eqref{eq:DT_NL_sys}.  
Suppose we have $x_0^1, ..,x_0^k$ initial conditions of the dynamical system, then for every initial condition, we have
\begin{align*}
x_n^i = T^n(x_0^i) \qquad \mbox{for}\; \;   i = 1,...,k.   
\end{align*}
The corresponding Koopman operator, trained on these trajectories will give the result 
\[
\Psi(x_{n}) = \mathbb{U}\Psi(x_{n-1})= [\mathbb{U}\circ\mathbb{U}]\Psi(x_{n-2})=\cdots = \underbrace{[\mathbb{U}\circ \cdots \circ \mathbb{U}]}_{n\; \mbox{ times}}\Psi(x_{0})= \mathbb{U}^n \Psi(x_0)
\]
Then the flow map is approximated by the Koopman projection on the state-observable as 
\[ 
T^n(x_0^i) = P_x \Psi_n(x_0^i),
\] 
where $P_x$ is the projection operator which maps from the observable space on to the state space. 

In the next section, we derive conditions to search for invariant subspaces starting locally from an invariant subspace of a dynamical system and develop the phase space stitching result to identify the global Koopman operator. 
\section{Global Phase Space Exploration Using Approximate Koopman Operator}
\label{sec:global_phase_space}

%


This section discusses the phase space analysis of a nonlinear system using the approximate Koopman operator obtained using time-series data from either a simulation or an experiment. 

%

\subsection{Phase Space Discovery}
\label{sec:phase_space_discovery}
We begin the analysis by considering the time-series data in an invariant subspace and computing the corresponding local Koopman operator. {We define conditions under which the Koopman operator must be updated as new spatial data is added. The Koopman operator for the full state space is eventually computed under the assumption that the spatial data is available from all the invariant subspaces in the system. Otherwise, constructing the global Koopman operator or identifying the global phase space is impossible without some understanding about the underlying system (refer to the discussion in Subsection} \ref{sec:phase_space_stitching} {where the symmetry information is assumed to be known).}

%

\begin{assumption}
Let $M_p \subset {\cal M}$ be an invariant subset of the phase space of the dynamical system \eqref{eq:DT_NL_sys}.  We suppose that each invariant subset $M_p$ has compact closure.
\end{assumption}
Then for the time-series data in $M_p$, the Koopman operator is given by ${\cal K}_p$. Let,  
\begin{align*}
    {\cal E}_p({\Psi}_p, n) :=  \max_{x_0 \in M_p} ||{\Psi}_p (T^n(x_0)) -  \mathcal{K}_{p}^n{\Psi}_p(x_0) || 
\end{align*}
where  $n$ denote the number of time-steps for which the initial condition is evolved. The notation, ${\cal K}_p^n$ indicates that Koopman operator ${\cal K}_p$ is raised to the power $n$. The observable functions denoted by $\Psi_p$ with an additional subscript $p$ to imply that these observables are defined on the subspace, $M_p$ of the state space. Since $M_p$ has a compact closure, ${\cal E}_p({\Psi}_p, n)$ exists and is finite for all $n \in \mathbb{N}$. 
Hence, there exists an $\varepsilon_p({\Psi}_p) \geq 0$ such that 
\begin{align}
    \lim_{n\to \infty} {\cal E}_p({\Psi}_p, n)\leq \varepsilon_{p}({\Psi}_p)
    \label{eq_Koopman_learning_error}
\end{align}
{The notation, ${\cal E}_p$ denote that the initial condition to compute this error is chosen from the subspace $M_p$.} In essence, for each subspace $M_p$ there is a finite upper-bound for the steady-state prediction error ${\cal E}_p({\Psi}_p, n).$  The magnitude of $\varepsilon_p({\Psi}_p)$ indicates the fitness of the Koopman operator for the underlying time-series data generated from a particular subspace $M_p$.  
%

\begin{assumption}
The invariant subspace, $M_p$ is a Koopman-invariant subspace, that is, the subspace formed by the span of observables is invariant under the action of the Koopman operator, ${\cal K}_p$. 
\label{ass:Koopman-invariant}
\end{assumption}
It follows from Assumption \ref{ass:Koopman-invariant}, $\varepsilon_p( \Psi_p) = 0$. 
{Assume the time-series data from the invariant subspace $M_p$ is given. Then, using one of the DMD variants explained in Section} \ref{sec:DMD_variants}  {the associated local approximate Koopman is computed from the time-series data. When new spatial time-series data emerges, let's assume it belongs to the subspace $M_q$. Now, it is natural to wonder how $M_q$ is related with $M_p$. Furthermore, under what conditions is a new approximate Koopman computed corresponding to the new spatial time-series data? The following result formally specifies when the Koopman operator must be updated.}


\begin{lemma}
Let the dynamical system \eqref{eq:DT_NL_sys} have $v$ invariant subspaces, $\{M_i\}_{i=1}^v$ such that $M_i \cap M_j = \Phi$ for $i\neq j$, where $\Phi$ is the null set. Suppose $M_p \subset {\cal M}$ be a Koopman-invariant subspace with the corresponding Koopman operator ${\cal K}_p$ trained on $\Psi_p$ such that span of $\Psi$ is invariant under ${\cal K}_p$. 
Then, given a set $M_q\subset M$ we have
\[{\cal E}_q( \Psi_p, n) \geq {\cal E}_p( \Psi_p,n).\]
\end{lemma}

\begin{proof} Depending on how the subspace $M_q$ is related to $M_p$, we have the following two cases. 

\textbf{Case 1: $M_q \subset M_p$.} In this case, ${\cal E}_q( \Psi_p, n) = {\cal E}_p( \Psi_p,n)$ since $M_q \subset M_p$ and span of $\Psi$ is invariant under ${\cal K}_p$.

\textbf{Case 2: $M_q \not\subset M_p$.} In this case, either
    \begin{enumerate}
        \item By Assumption \ref{ass:Koopman-invariant}, ${\cal E}_q( \Psi_p, n) \le {\cal E}_p( \Psi_p,n) \equiv 0$ would imply that ${\cal E}_q( \Psi_p, n) \equiv 0$    and in this case, the dictionary functions $ \Psi_p$ is valid for the subspace $M_q$ (or)
        \item ${\cal E}_q( \Psi_p, n) > {\cal E}_p( \Psi_p,n)$ and in this case, a new dictionary needs to be defined to compute ${\cal K}_q$ for subspace $M_q$.
    \end{enumerate}
%
 
%
%
\label{lemma_K_learning_error}
\end{proof}
%



\begin{remark}
If additional temporal time-series data is obtained in contrary to the spatial time-series data as discussed in Lemma \ref{lemma_K_learning_error}, then recursive least squares is used to update the Koopman operator as shown in \cite{sinha2020data}. 
\end{remark}


{Hitherto we have seen how to update the Koopman operator when new spatial or temporal time-series data emerges. In the following, we present conditions under which a new invariant subspace is discovered using the spectral properties of the updated approximate Koopman operator and the local approximate Koopman operator corresponding to $M_p$. }


\begin{proposition}
Let the subspace ${M}_p \subset {\cal M}$ be the smallest invariant subspace of ${\cal M}$ and the corresponding Koopman operator is given by ${\cal K}_{p}$. Choose ${M}_q$ such that ${M}_p \subset {M}_q \subseteq {\cal M}$ and denote the Koopman operator corresponding to ${\cal M}_q$ by ${\cal K}_q$. Then the discovery of new invariant subspaces in the phase space of the original dynamical system is such that the difference $||{\cal K}_{q}||_{\rho} - ||{\cal K}_{p}||_{\rho}$ increases monotonically where $\parallel \cdot \parallel_{\rho}$ denote the geometric multiplicity of the unitary eigenvalue, $\lambda = 1$, that is, number of distinct eigenvectors associated with the stable eigenvalue, $\lambda = 1$. 
\label{prop:discovery_inv_space}
\end{proposition}
\begin{proof}
The proof follows by noting that for a subspace ${\cal M}_q \supset {\cal M}_p$, the inequality $\parallel {\cal K}_q \parallel_{\rho} \geq \parallel {\cal K}_p \parallel_{\rho}$ is always satisfied.
\end{proof}

{By continually applying the result, Proposition} \ref{prop:discovery_inv_space}, {to data across the state space, the Koopman operator corresponding to the complete state space may be computed. When the invariant subspaces are unknown but the data is available, this discovery is especially valuable. In the following section, we will start with the global phase space and show how the spectral properties of the Koopman operator can be utilized to partition the global phase space into invariant subspaces. }

%
\subsection{Invariant Phase Space Decomposition Using Global Koopman Operator}
\label{sec:state_space_partition}
{In this section, we will explore how to use the global Koopman operator to investigate a dynamical system \emph{locally}. For the self-containment of this work, we recollect certain relevant concepts from} \cite{budivsic2012applied,petersen1989ergodic}. {The findings below are presented in terms of a finite dimensional approximate Koopman operator, although they are universal and can be applied to infinite dimensional Koopman operators. Interested readers may refer} \cite{budivsic2012applied,petersen1989ergodic} and the references contained therein for more information.

Suppose $T: {\cal M} \to {\cal M}$ is an ergodic transformation, that is, for any measurable set ${\cal S} \subset {\cal M}$ invariant under $T$, that is,  $T^{-1}({\cal S}) = {\cal S}$, then we either have  $\mu({\cal S}) = 0$ or $\mu({\cal S}) = 1$, and all the eigenvalues of ${\cal K}$ are simple (that is, the algebraic multiplicity, $m_a$ of eigenvalues is 1) \cite{budivsic2012applied,petersen1989ergodic}. However, if $T$ is not ergodic, then the state space can be partitioned into subsets ${\cal S}_i$ (minimal invariant subspaces) such that the restriction $T|_{{\cal S}_i}: {\cal S}_i \to {\cal S}_i$ is ergodic. A partition of the state space into ergodic sets is called an ergodic partition or stationary partition \cite{budivsic2012applied}.

Furthermore, all ergodic partitions are disjoint and they support mutually singular functions from ${\cal G}$ \cite{budivsic2012applied}. Therefore, the number of linearly independent eigenfunctions of ${\cal K}$ corresponding to an eigenvalue $\lambda$ is bounded above by the number of ergodic sets in the state space \cite{budivsic2012applied}. Furthermore, for non-ergodic systems with attractor sets (equilibrium points and limit cycles), the number of unit eigenvalues of $\cal K$ is equal to the number of attractor sets of the underlying system \cite{budivsic2012applied}. With this we have the following.

%
\begin{lemma}
Let $\lambda$ be an eigenvalue of the Koopman operator ${\cal K}$. Suppose the algebraic multiplicity ($m_a$) of the eigenvalue $\lambda$, is equal to the geometric multiplicity ($m_g$), $v$, then the corresponding $v$ eigenfunctions are linearly independent. Moreover the $v$ linearly independent functions map to at most $v$ invariant subspaces in the state space.  
\label{lemma:am_gm} 
\end{lemma}
\begin{proof}
Corresponding to the Koopman eigenvalue, $\lambda$, if $m_a = m_g = v$, 
then there exists $v$ linearly independent eigenvectors\footnote{Without loss of generality, the terms eigenfunctions and eigenvectors are used interchangeably in this manuscript unless otherwise mentioned.} (follows from standard results on matrices \cite{horn2012matrix}). 
The rest of the proof follows from \cite{budivsic2012applied}.
\end{proof}


 
\begin{figure}[h]
    \centering
    \includegraphics[width = 0.75 \columnwidth]{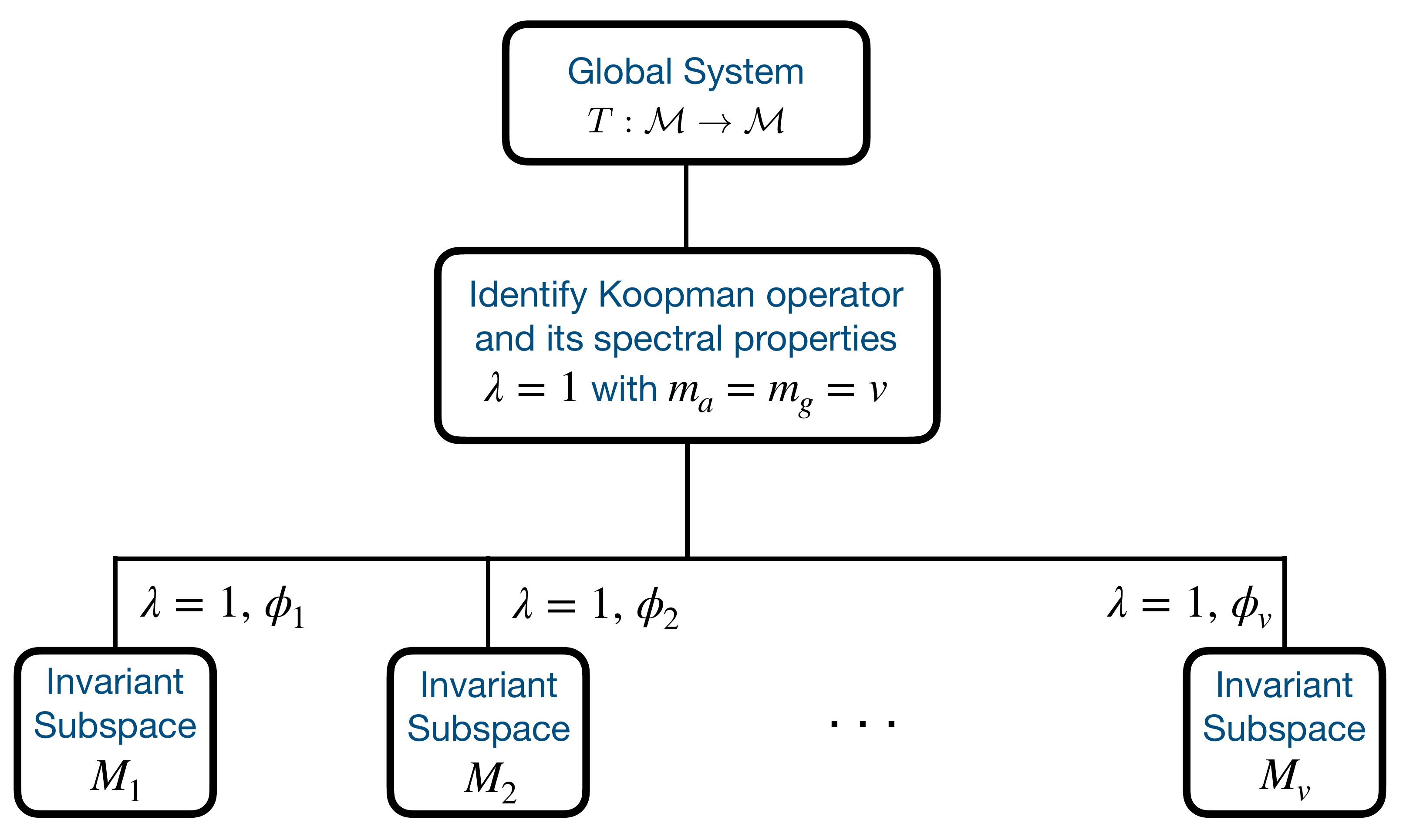}
    \caption{Summary of the invariant phase space decomposition using the global Koopman operator. The nonlinear system is denoted by $T$ and the spectral properties of its corresponding Koopman operator are studied. The eigenvalues of the Koopman operator are denoted by $\lambda$ where the unit eigenvalue appears with $m_a = m_g = v$ and $\phi_1, \phi_2, \dots, \phi_v$ denote the respective eigenfunctions.}
    \label{fig:phase_space_exploration}
\end{figure}
The results discussed in this subsection to identify the invariant subspaces are summarized in Fig. \ref{fig:phase_space_exploration}. The next section describes how the evolution of system on the complete phase space (i.e., a global Koopman operator) can be identified if the system evolution on invariant subspaces is known (i.e., local Koopman operators are known). 

\subsection{Coupling Distinct Phase Representations Assuming Approximate Koopman Operators}
\label{sec:phase_space_stitching}
Consider the dynamical system \eqref{eq:DT_NL_sys} and let $\{M_i\}_{i=1}^v$ be the invariant sets of the dynamical system. For any $M_p$, let 
\[{\Psi}_{p}(x) = \left[\psi_{1}^p(x), \cdots , \psi_{n_p}^p(x)\right] \]
be the set of dictionary functions used for computation of the local Koopman operator ${\cal K}_p$ where $\psi_{j}^p:{\cal M} \to \mathbb{C}$ is a scalar-valued function for every $j \in \{1,2,\dots, n_p\}$. Then these local Koopman operators can be combined to form a single Koopman operator which we refer to as the \textit{stitched Koopman operator} and it is given by 
\begin{align}
     {\cal K}_{\cal S} = \mbox{diag}({\cal K}_{1},{\cal K}_2,\cdots,{\cal K}_{v})
     \label{eq:stitched_Koopman}
\end{align}
with 
\begin{align}
    {\Psi}(x) = \begin{bmatrix} \chi_1(x) { \Psi}_{1}(x) \\ \chi_2(x) { \Psi}_{2}(x) \\ \vdots \\ \chi_{v}(x) { \Psi}_{v}(x) \end{bmatrix}
    \label{eq:stitched_observables}
\end{align}
where $\chi_i(x)$ is the characteristic function corresponding to the (invariant) subspace $M_i$ and it is defined as follows
\begin{align*}
    \chi_i(x) = \begin{cases} 1, & \mbox{if}\;\ x \in M_i \\ 
    0, & \mbox{otherwise}. \end{cases}
\end{align*}
for $i=\{1,\cdots,v\}$. 
%
The order of these local Koopman operators while forming the stitched Koopman operator doesn't matter as long as the local Koopman operators are stacked in accordance with their corresponding observable functions. Note that the intersection of any two invariant sets is a set of measure zero and such sets have not been considered in the above result. The stitched Koopman operator is a block diagonal matrix with ${\cal K}_{\cal S}\in \mathbb{R}^{L\times L}$ where $L = \sum_{i=1}^{v}n_i$. The idea behind the phase space stitching algorithm is depicted in Fig. \ref{fig:phase_space_stitching}. 
\begin{figure}[h!]
    \centering
    \includegraphics[width=0.75 \columnwidth]{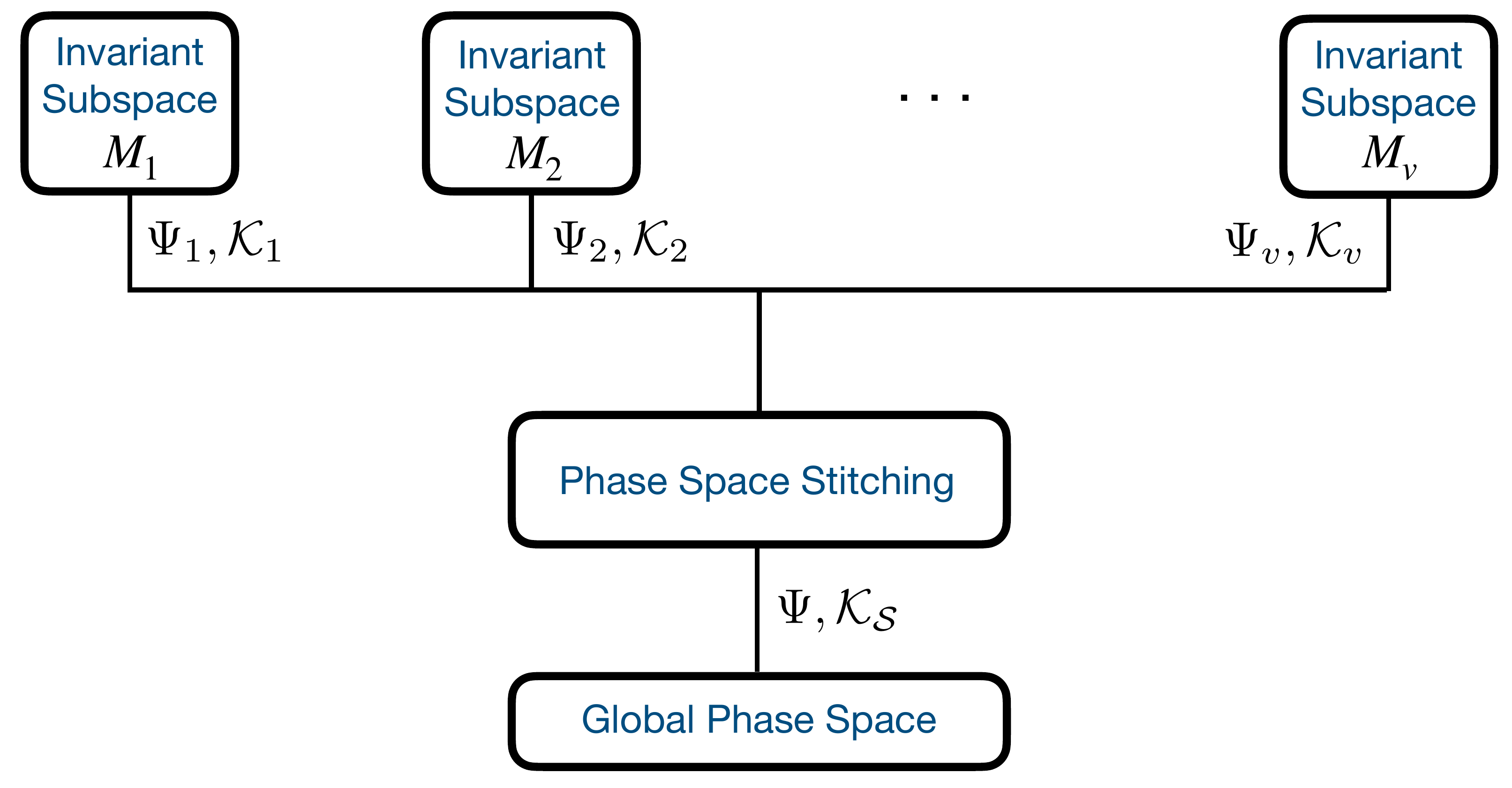}
    \caption{Summary of phase space stitching algorithm using Koopman operators. For every invariant subspace $M_i$, \; $ i \in \{1,2,\dots,v\}$, the corresponding obervable functions and the (local) Koopman operator are denoted by $ \Psi_i, {\cal K}_i$. The local Koopman operators from each invariant subspace are combined to form the stitched Koopman, ${\cal K}_{\cal S}$ with the corresponding observables $ \Psi$ using Algorithm \ref{alg}.}
    \label{fig:phase_space_stitching}
\end{figure}
One of the key differences between identifying the global Koopman operator as described in Section \ref{sec:phase_space_discovery} and \ref{sec:phase_space_stitching} is summarized in the remark given below. 
\begin{remark}
In Section \ref{sec:phase_space_discovery}, the global Koopman operator is identified by starting locally in an invariant subspace and with new time series data, Lemma \ref{lemma_K_learning_error} dictates when a new Koopman operator needs to be identified and Proposition \ref{prop:discovery_inv_space} shows when a larger invariant space is discovered. This process is repeated until the time-series data from the entire space is seen and the corresponding global Koopman operator has been determined. 
This is however different from the (global) stitched Koopman operator (Section \ref{sec:phase_space_stitching}) which assumes knowledge of a local Koopman operator that defines an Koopman invariant subspace for each invariant subspace of the phase space. These results can be readily applied for time-series data generated by multiple experiments as shown in second-order system examples in Section \ref{sec:simulation}. 
\end{remark}

{So far, it has been assumed that every invariant subspace's time-series data is known. This assumption is relaxed in the next subsection thanks to the knowledge of underlying dynamical system's symmetry properties. In particular, we explore the global phase space analysis of a class of dynamical systems with symmetry known as equivariant systems.}

\subsection{Global Koopman Operator for Equivariant Dynamical Systems:}
\label{sec:equivariant_systems}
In this work, we show how to leverage the phase space stitching result to compute the global Koopman for equivariant systems. We begin with the definition of equivariant dynamical systems. algorithm for any system and in particular the equivariant dynamical system. 

\begin{definition}[Equivariant Dynamical System]
Consider the dynamical system \eqref{eq:DT_NL_sys} and let $G$ be a group acting on ${\cal M}$. Then the system  \eqref{eq:DT_NL_sys} is called $G$-equivariant if 
$T(g\cdot x) = g\cdot T(x) \;\text{for}\; g\in G.$
\end{definition}
Koopman operator theory for equivariant systems is introduced in \cite{sinha2020equivariant} and we only recall a key result that is relevant to the global Koopman identification for equivariant systems.  
In an equivariant system, if the finite dimensional approximation of the Koopman operator is known in one invariant subspace, the Koopman operator for other invariant subspaces can be identified provided the symmetry between the invariant subspaces is known.  The ensuing result establishes this. 
\begin{theorem}\label{K_i_K_j_theorem}
Let ${\cal M}_p$ be an invariant set of the $G$-equivariant system (\ref{eq:DT_NL_sys}) and let ${\cal K}_p\in\mathbb{R}^{n_p\times n_p}$ be the local Koopman operator on ${\cal M}_p$ with dictionary function ${\Psi}_p(x)$, $x\in{\cal M}_p$. Suppose for $g \in G$, $g \cdot {\cal M}_p\subset {\cal M}_q$ and ${\cal K}_q$ be the local Koopman operator on ${\cal M}_q$ with dictionary functions ${\Psi}_q={\Psi}_p$. Then 
\[{\cal K}_p = \gamma {\cal K}_q \gamma^{-1},\]
where $g \mapsto \gamma \in \Gamma$ and $\Gamma$ is the $n_p$ dimensional matrix representation of $G$ in $\mathbb{R}^{n_p}$.
\end{theorem}
\begin{proof}
We refer the readers to \cite{sinha2020equivariant} for the proof. 
\end{proof}

Under the assumption that the symmetry of the system is known, the local Koopman operator in every invariant subspace of the system can be identified starting with the knowledge of time-series data in only one invariant subspace. This idea of computing the global Koopman for equivariant systems is portrayed in Fig. \ref{fig:phase_space_stitching_equivariant}. In the following, we present the phase space stitching result as an algorithm for any system and in particular the equivariant dynamical system. 

\begin{figure}[h!]
    \centering
    \includegraphics[width=0.75 \columnwidth]{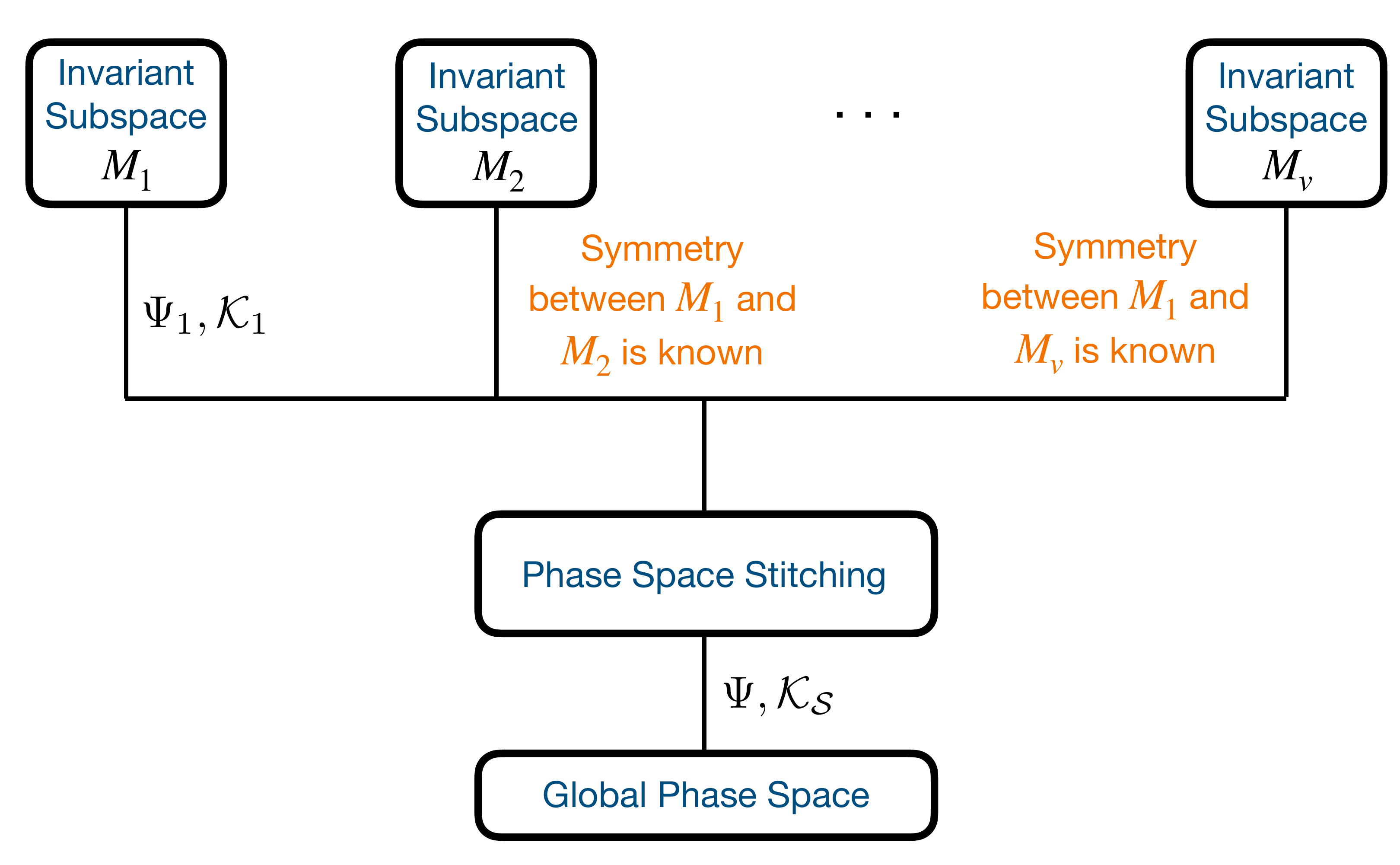}
    \caption{Phase space stitching for equivariant systems. The symmetry information between the invariant subspaces $M_i$, \; $ i \in \{1,2,\dots,v\}$ is known and the local Koopman operator is known only in $M_1$. Theorem \ref{K_i_K_j_theorem} is applied to identify the local Koopman operator in every invariant subspace and they are now combined to form the stitched Koopman, ${\cal K}_{\cal S}$ with the corresponding observables $ \Psi$ using Algorithm \ref{alg}.}
    \label{fig:phase_space_stitching_equivariant}
\end{figure}


\begin{algorithm}[htp!]
\caption{Phase space stitching to compute global Koopman operator:}
\begin{enumerate}
\item \textbf{Input:} For every invariant subspace ${\cal M}_p$, where $i\in \{1,2,\hdots,v\}$, the observables $ \Psi_p(x)$ and their corresponding approximate Koopman operator ${\cal K}_p$ 
\begin{itemize}
    \item \textit{For an equivariant system}: Only one invariant subspace (${\cal M}_p$), its corresponding observable functions ($ \Psi_p(x)$), the approximate local Koopman operator (${\cal K}_p$) and the symmetry between the given invariant subspace with the other invariant subspaces. 
\end{itemize}
%
%
\item Define the observables for the entire space (obtained from the union of all invariant subspaces) and form the stitched Koopman operator
 \begin{align*}
    {\Psi}(x) = \begin{bmatrix} \chi_1(x) { \Psi}_{1}(x) \\ \chi_2(x) {\Psi}_{2}(x) \\ \vdots \\ \chi_{v}(x) { \Psi}_{v}(x) \end{bmatrix}, \qquad {\cal K}_{\cal S} = \mbox{diag}({\cal K}_{1},{\cal K}_2,\cdots,{\cal K}_{v})
\end{align*}
\begin{itemize}
    \item For an equivariant system, the local Koopman operators corresponding to each invariant subspace are computed applying Theorem \ref{K_i_K_j_theorem}. 
\end{itemize}
%
\item \textbf{Output:} The stitched Koopman now gives the global approximate Koopman operator. 
\end{enumerate}
\label{alg}
\end{algorithm}

{In this section, we explored the phase space of a dynamical system using the spectral properties of the approximate Koopman operator to discover invariant subspaces, partition the global phase space into multiple invariant subspaces, and fuse local Koopman operators to construct the global Koopman. However, in the following section, we explore the phase space of one of two topologically conjugate dynamical systems given knowledge of the phase space of the other.}


\section{Topologically Conjugate Dynamical Systems}
\label{sec:TC_systems}
If the global phase space has a topological conjugacy, then the global Koopman for the topological conjugate system may be determined, as well as the system's characteristics such as the number of attractors.
We begin with the definition of topologically conjugate dynamical systems. 
\begin{definition}[Topological Conjugacy]
Let $T_1:{\cal M} \to {\cal M}$ and $T_2:{\cal N} \to {\cal N}$ describe two nonlinear dynamical systems. Then $T_1$ and $T_2$ are topologically conjugate if and only if there exists a homeomorphism $h:{\cal N} \to {\cal M}$ such that  $T_1 = h \circ T_2 \circ h^{-1}$. Then the homeomorphism $h$ is called a topological conjugacy between $T_1$ and $T_2$. 
\end{definition}
The following result establishes the relation between the observables of topologically conjugate systems. 
%

\begin{theorem}
Let $T_1:{\cal M} \to {\cal M}$ and $T_2:{\cal N} \to {\cal N}$ be two topologically conjugate dynamical systems with the homeomorphism $h:{\cal N} \to {\cal M}$. Let $\Psi$ be the basis function on $\cal M$, such that $\mathbb{U}_{\Psi}$ is the representation of the corresponding Koopman operator for $(T_1,{\cal M})$ with respect to $\Psi$. Suppose if the observables on $(T_2,{\cal N})$ are $\Theta = {\Psi}\circ h$, then the representation of the Koopman operator for $T_2$ with respect to $\Theta$ (denoted  by $\mathbb{U}_{\Theta}$) satisfies $\mathbb{U}_{\Psi}=\mathbb{U}_{\Theta}$.
\label{thm:observable_functions_TC}
\end{theorem}
\begin{proof}
Given a $x \in {\cal M}$ and $y \in {\cal N}$, then they are related by the homeomorphism as $x = h(y)$. Suppose $\mathbb{U}_{\Psi}$ be the Koopman operator trained on the observables $\Psi$ defined on the state space of $T_1$. 
Then, we have
\begin{align*}
    [\mathbb{U}_{\Psi} \Psi](x) & = \Psi(T_1(x)) = \Psi(T_1(h(y))) = \Psi(h(T_2(y))) = [\mathbb{U}_{\Theta} (\Psi \circ h)](y) = [\mathbb{U}_{\Theta} \Theta](y)
\end{align*}
where $\mathbb{U}_{\Theta}$ is the Koopman operator corresponding to the system $T_2$ defined on the space of observables given by $\Theta = \Psi \circ h$. Since we have $\Psi(x) = \Psi(h(y)) = \Theta(y)$, we obtain $\mathbb{U}_{\Psi} = \mathbb{U}_{\Theta}$ and hence the proof. 
\end{proof}



\begin{figure}[h!]
    \centering
    \includegraphics[width = 0.95 \textwidth]{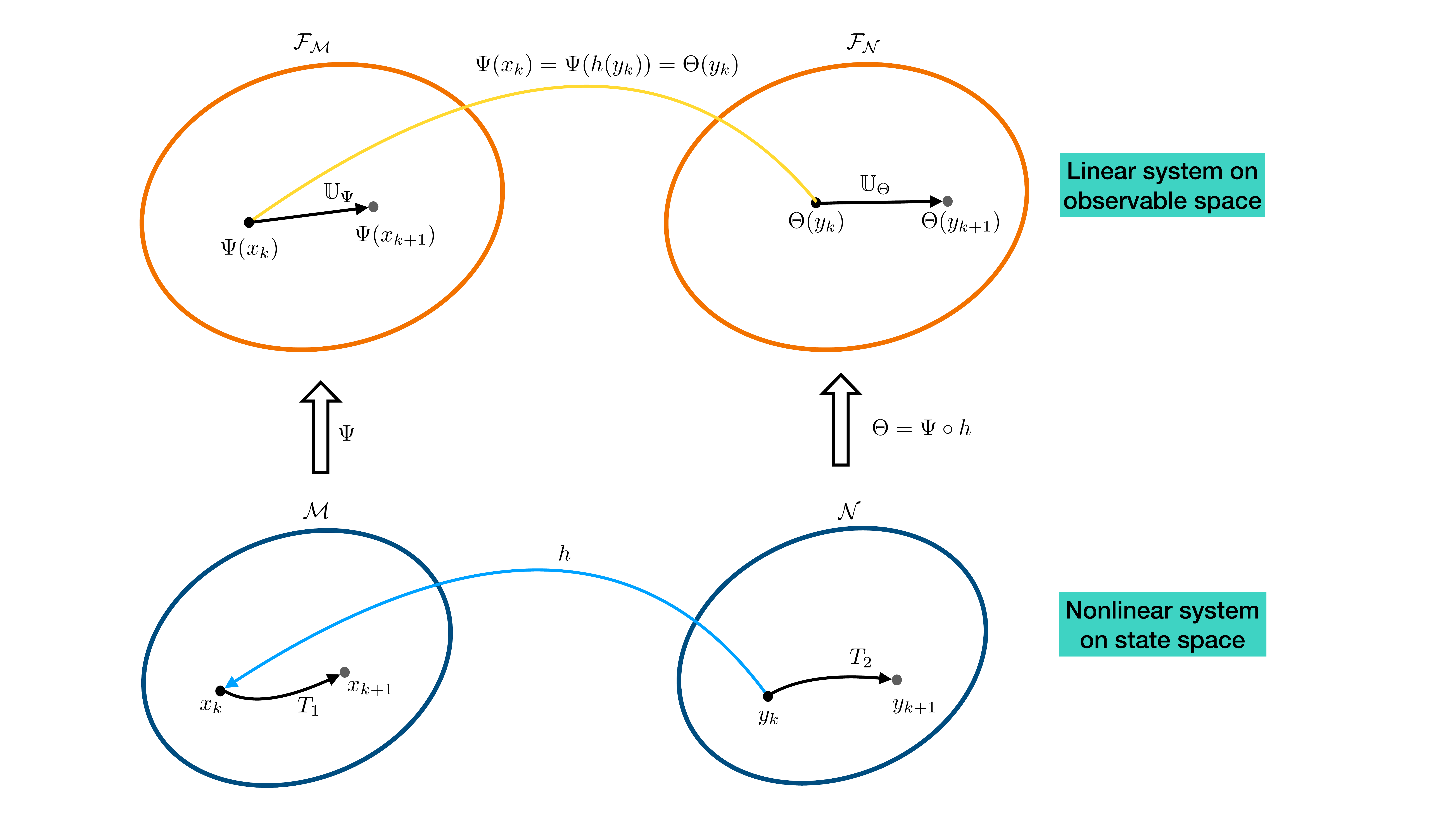}
    \caption{Overview of the study of topologically conjugate systems using Koopman operator theory. The bottom sketch shows the evolution of nonlinear system on the state space and how homeomorphic function connects the systems $T_1$ and $T_2$. The top sketch illustrates the evolution of the system on the space of observables given by the Koopman operator and shows the relation between the spaces ${\cal F}_{\cal M}$ and ${\cal F}_{\cal N}$. The space ${\cal F}_{\cal M}$ and ${\cal F}_{\cal N}$ denote the space of observable functions defined on ${\cal M}$ and ${\cal N}$ respectively.}
    \label{fig:tc_overview}
\end{figure}

The above theorem  shows how the Koopman operators corresponding to the topologically conjugate systems are related in the special case where the observables are defined using the homeomorphism as shown in Theorem \ref{thm:observable_functions_TC}. However, in general a natural question that arises here is to understand the relation between the Koopman eigenvalues and Koopman eigenfunctions for the topologically conjugate systems. The following result formalizes this relation. 
%


\begin{proposition}
Suppose $T_1:{\cal M} \to {\cal M}$ and $T_2:{\cal N} \to {\cal N}$ describe two nonlinear dynamical systems that are topologically conjugate with the homeomorphism $h:{\cal N} \to {\cal M}$. The corresponding Koopman operators on $T_1$ and $T_2$ are given by $\mathbb{U}_{1}$ and $\mathbb{U}_{2}$. Then the eigenvalues and eigenfunctions of $\mathbb{U}_1$ and $\mathbb{U}_2$ are in one-to-one correlation with one another via the homeomorphism $h$.
\label{prop:eigenfunctions_TC}
\end{proposition}
\begin{proof}
Let $\Lambda_k$ and $F_{k}$ respectively denote the set of eigenvalues and set of eigenfunctions respectively for $\mathbb{U}_{k}$ where $k = \{1,2\}$. Suppose $x \in {\cal M}$,  $y \in {\cal N}$, and they are related by the homeomorphism by $x = h(y)$.  
%

Consider an eigenvalue $\lambda \in \Lambda_1$, and the corresponding eigenfunction $\phi \in F_1$. Then, 
\begin{align*}
     \lambda (\phi \circ h) (y) = \lambda \phi(x) = [\mathbb{U}_{1} \phi](x) = \phi(T_1(x)) = \phi(T_1(h(y))) = \phi(h(T_2(y))) = [\mathbb{U}_{2} (\phi \circ h)](y) 
\end{align*}
%
%
It now follows that $\lambda \in \Lambda_2$ and $\phi \circ h \in F_2$. Similarly for $\bar{\lambda} \in \Lambda_2$ and $\bar{\phi} \in F_2$, it can be shown that, $\bar{\lambda} \in \Lambda_1$ and $h^{-1} \circ \bar{\phi} \in F_1$ and hence it can be seen that the eigenvalues and eigenfunctions for topologically conjugate systems are in one-to-one correlation. 
\end{proof}

Note that the above result under the implicit assumption of eigenvalues being same for the Koopman operators corresponding to topologically conjugate systems is presented in \cite{budivsic2012applied}. 
%
%
More generally, in Proposition \ref{prop:eigenfunctions_TC}, we establish the relation between the eigenvalues and the eigenfunctions of the topologically conjugate systems. The ensuing result shows that the Koopman modes for both the topologically conjugate systems remain same. 
\begin{proposition}
Suppose $T_1:{\cal M} \to {\cal M}$ and $T_2:{\cal N} \to {\cal N}$ describe two nonlinear dynamical systems that are topologically conjugate with the homeomorphism $h:{\cal N} \to {\cal M}$. Then the Koopman modes for both the topological conjugate systems are same. 
\label{prop:Koopman_modes_TC}
\end{proposition}
\begin{proof}
Let $x_0 \in {\cal M}$, $y_0 \in {\cal N}$ and they are related to each other through the homemorphic function, such that, $x_0 = h(y_0)$. 
Suppose ${\cal G}$ be the space of scalar-valued observable functions defined on the state space ${\cal M}$. Let ${f} = \left[f_1 \; f_2 \; \hdots \; f_m \right]$ be a vector-valued observable function where each $f_i \in {\cal G}$ for $i\in \{1,2,\dots, m\}$, then ${f}$ can be expressed in terms of the Koopman tuple as follows. 
\begin{align}
    {f}(x_t) = & \sum_{j=1}^{\infty} \lambda_j^t \phi_j(x_0) {\vartheta}_j \quad \mbox{for} \; x_t \in {\cal M} 
    \label{eq:KMD_expansion_TC}
\end{align}
The KMD shown in Eq. \eqref{eq:KMD_expansion_TC} can be rewritten as 
\begin{align*}
    {f}(h(y_t)) = & \sum_{j=1}^{\infty} \lambda_j^t \phi_j(h(y_0)) {\vartheta}_j \\
    ({f}\circ h)(y_t) = & \sum_{j=1}^{\infty} \lambda_j^t (\phi_j \circ h) (y_0) {\vartheta}_j
\end{align*}
From Theorem \ref{thm:observable_functions_TC} and Proposition \ref{prop:eigenfunctions_TC}, we know that the observables for the topologically conjugate system are given by ${f}\circ h$ and the eigenfunction for the topologically conjugate system is given by $\phi \circ h$. Therefore, the Koopman modes, ${\vartheta}_j$ remain same for both the topologically conjugate systems.   
\end{proof}

Note the Koopman modes in Eq. \eqref{eq:KMD_expansion_TC} is a function of the choice of the vector-valued observables, $ f$. 
{We observe that if time-series data from one system and its topological conjugacy are known, the global phase space of the topologically conjugate system can be investigated, that is, the number of invariant subspaces, attractor sets, spatiotemporal modes, and so on.
It's important to note that the symmetry properties of an equivariant system are shared by the system's invariant subspaces, while topological conjugacy is shared between separate systems. A special case arises in an equivariant system if each of the invariant subspaces can be related by a homeomorphic function with other invariant subspaces. We may thus claim that the invariant subspaces are topologically conjugate in this scenario.}
%
%
%

\section{Simulation Study}
\label{sec:simulation}
We consider several second-order nonlinear dynamical systems in this section and use the results from Section \ref{sec:global_phase_space} and Section \ref{sec:TC_systems} to demonstrate the invariant phase space identification, phase space stitching, global Koopman operator computation with and without symmetric properties and phase space identification of topologically conjugate dynamical systems. The phase space study of the bistable toggle switch comes first.

%
%
\subsection{Bistable Toggle Switch}
The genetic toggle switch is a seminal memory device developed by Collins and Gardner \cite{gardner2000construction} to simulate binary logic and memory inside of living cells.  The design and analysis of a toggle switch model has been the subject of many studies \cite{tian2006stochastic,munsky2010guidelines,yeung2021data}.  In our case, we use the toggle switch as simple example of a two state nonlinear system with multiple equilibria and a non-trivial invariant subspace partition.  In it's most basic form, the toggle switch consists of two biological states, usually proteins, that repress or attenuate each other to apply mutual negative feedback.  The toggle switch's dynamics can be described by the following governing equations \cite{gardner2000construction}:
\begin{equation}
\begin{aligned}
\dot{x}_1 = & \frac{\alpha_1}{1+x_2^{\beta}} - \kappa_1 x_1 \\
\dot{x}_2 = & \frac{\alpha_2}{1+x_1^{\gamma}} - \kappa_2 x_2
\end{aligned}
\label{eq:bistable_toggle_switch}
\end{equation}
where the states $x_1, x_2 \in \mathbb{R}$ indicate the concentration of the repressors $1$ and $2$. The effective rate of synthesis for repressors 1 and 2 are denoted by $\alpha_1$ and $\alpha_2$. The self decay rates of the concentration of repressors 1 and 2 are  given by $\kappa_1>0$ and $\kappa_2>0$. The cooperativity of repression of promoter 1 and 2 are respectively $\gamma$ and $\beta$. The bistable toggle switch is mathematically modelled to investigate a bistable gene-regulatory network \cite{gardner2000construction}. There are two stable equilibrium points in this system (indicating bistability).  When the parameters like $\alpha_1, \alpha_2$ for gene concentrations, $x_1, x_2$, are different, a monostable equilibrium is observed \cite{gyorgy2016quantifying}. 

%

\textbf{Data Generation:}  
We consider $81$ initial conditions in the global state space and for each initial condition, system is evolved for $21$ time points. The complete dataset is considered and the centers ($c_i$) of each cluster are determined using \textit{k-means} clustering. The Gaussial radial basis functions are then defined using the centers, $c_i$. As indicated in Eq. \eqref{edmd_op}, the related Koopman operator is computed using the EDMD algorithm by minimizing the residue function. 
The size of the dictionary is chosen to be $30$ and each dictionary function is of the form, $\psi_i(x) = \exp(-{\parallel x - c_i \parallel^2/\sigma^2})$ for $i\in\{1,2,\dots, 30\}$ where $\sigma$ is chosen to be $0.4$. Figure \ref{fig:bistable_phase_portrait} shows the phase portrait of this system with two stable equilibrium points (i.e., two stable attractors). The equilibrium points for the bistable toggle switch system are given by $(2,0.16)$ and $(0.16, 2)$. 
\vspace{-0.3 cm}
\begin{figure}[h!]
\begin{center}
\includegraphics[width=0.5 \textwidth]{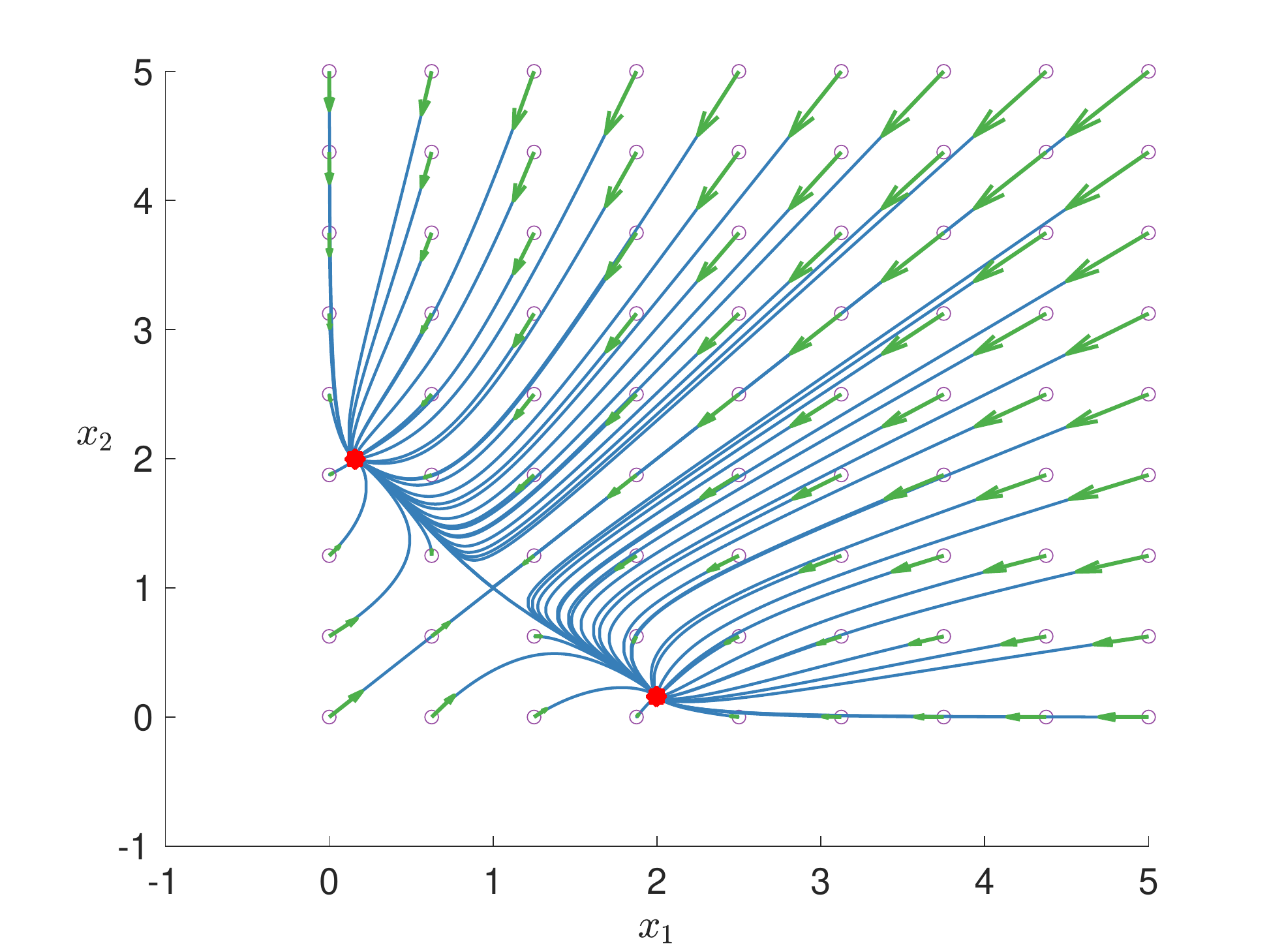}
\caption{Phase portrait of the bistable toggle switch. This system has two stable attractors for device parameters: $\beta = 3.55$, $\gamma = 3.53$, $\alpha_1 = \alpha_2 = 1$ and $\kappa_1 = \kappa_2 = 0.5$. Circles indicate the initial conditions, green arrow indicate the direction of the vector field at the initial condition and the red dots denote the equilibrium points.}
\label{fig:bistable_phase_portrait}
\end{center}
\end{figure}



\textbf{Phase Space Exploration (Global to Local):} 
The eigenvalues of the global Koopman operator ${\cal K}$ are shown in Fig. \ref{fig:eigenvalues_bistable} and there are two (dominant) eigenvalues at $\lambda = 1$. The eigenvectors corresponding to these eigenvalues are linearly independent and hence, it shows that there are two invariant subspaces on the state space (follows from Lemma \ref{lemma:am_gm}). Moreover, these eigenvectors associated with the unitary eigenvalues when evaluated on the state space captures the attractors whose centers are the equilibrium points as seen in Fig. \ref{fig:bistable_inv_spaces}. 
\begin{figure}[h!]
\begin{center}
\includegraphics[width=0.6 \textwidth]{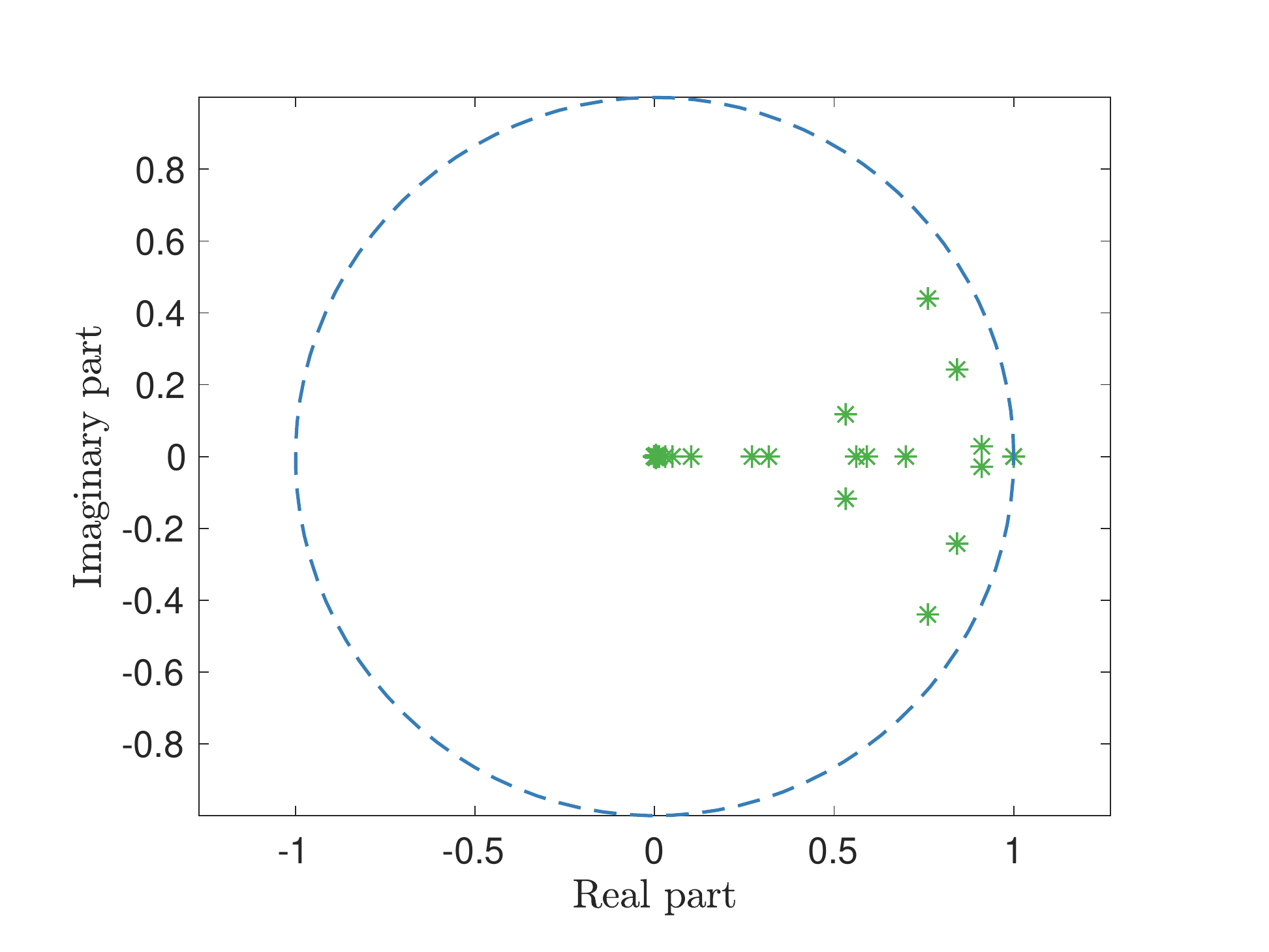}
\caption{Eigenvalues of the Koopman operator ${\cal K}$ corresponding to the bistable toggle switch system given in Eq. \eqref{eq:bistable_toggle_switch}.}
\label{fig:eigenvalues_bistable}
\end{center}
\end{figure}
\begin{figure}[h!]
\begin{center}
\subfigure[]{\includegraphics[width = 0.42 \linewidth]{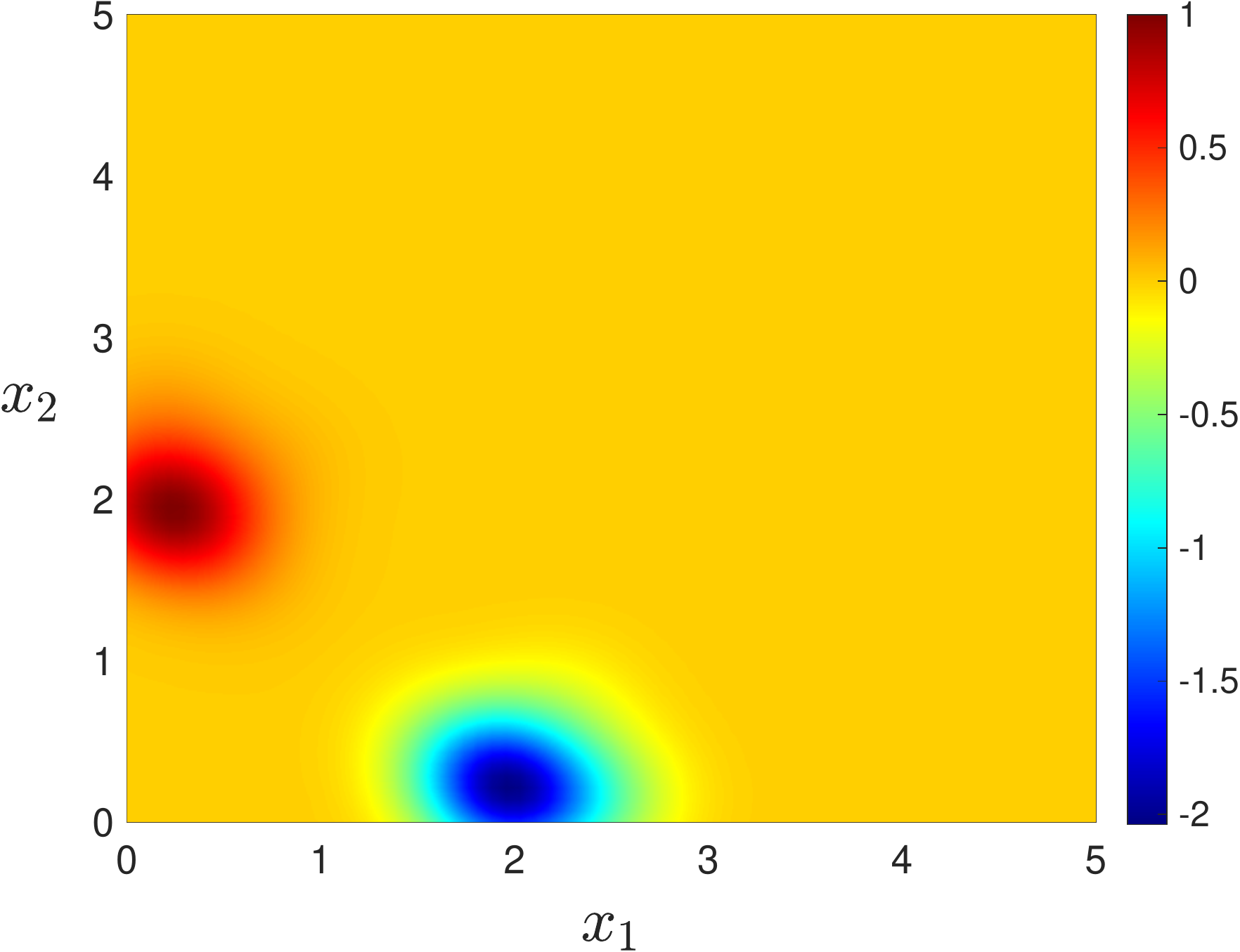}}
\subfigure[]{\includegraphics[width = 0.42 \linewidth]{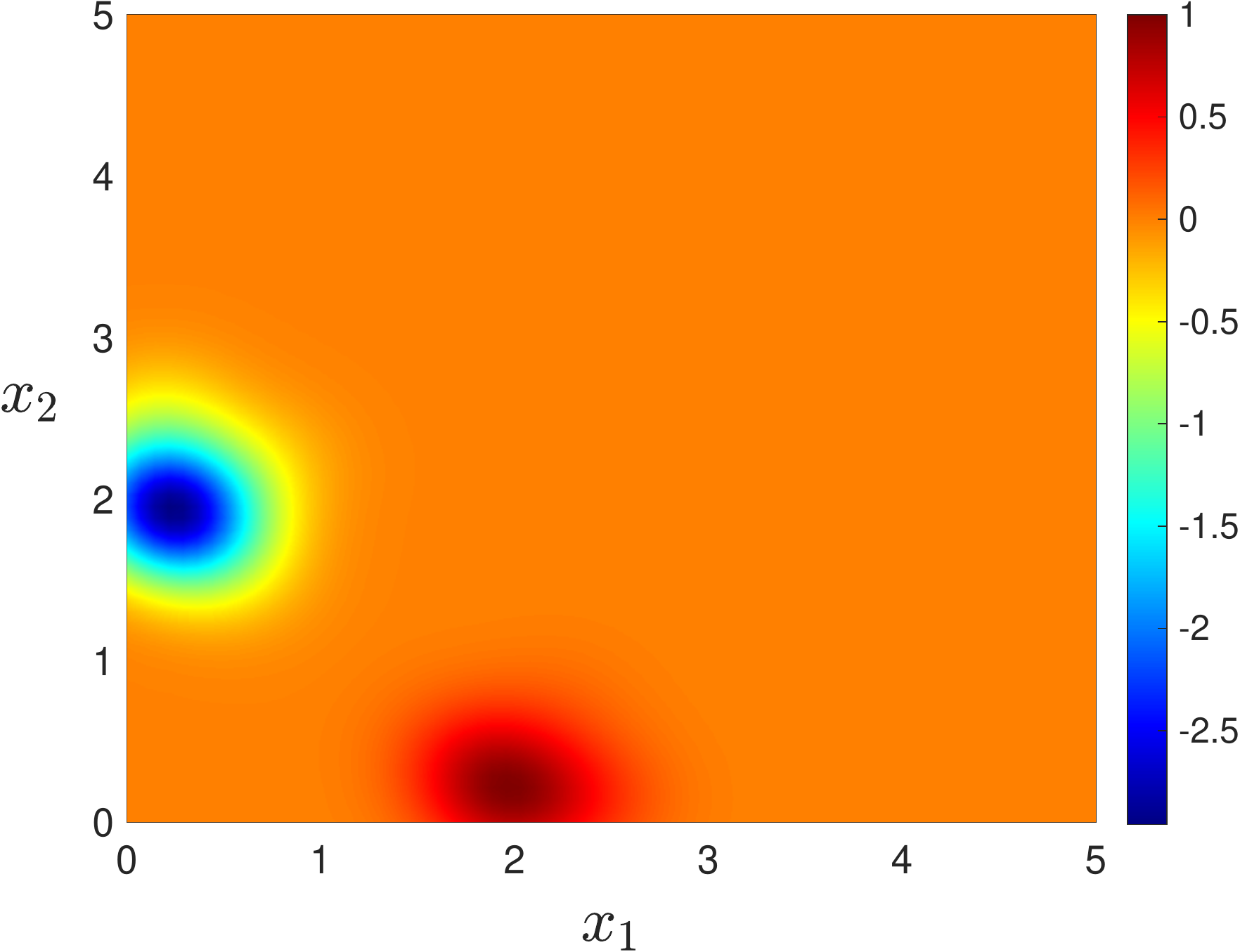}}
\caption{(a) and (b) Eigenvectors of the Koopman operator corresponding to the eigenvalue $\lambda = 1$ on the state space. The region around the equilibrium point can be seen inside the (blue) colored ellipses.}
\label{fig:bistable_inv_spaces}
\end{center}
\end{figure}

\textbf{Phase Space Exploration (Local to Global):} Assuming knowledge of the phase space on the local invariant subspaces, local Koopman operators are computed and the stitched Koopman operator corresponding to the global phase space is computed using the phase space stitching result (see Algorithm \ref{alg} and Section \ref{sec:phase_space_stitching}). Recall, this stitched Koopman is equivalent to the global Koopman operator but evolves on a different observable space. For computing each of the local Koopman operators, namely, ${\cal K}_{left}$ and ${\cal K}_{right}$, $30$ Gaussian radial basis functions are used. Clearly there is an eigenvalue at $\lambda = 1$ for ${\cal K}_{left}$ as well as ${\cal K}_{right}$ and their corresponding eigenvectors on the state space reveals the attractor sets around the equilibrium points. The stitched Koopman computed using Algorithm \ref{alg} is given by ${\cal K}_{\cal S}:=$diag$({\cal K}_{left}, {\cal K}_{right})$. The eigenvalues of the stitched Koopman operator are shown in Fig. \ref{fig:bistable_eigenvalues_stitched}. It is important to note that the size of ${\cal K}_{\cal S}$ is $60\times 60$ whereas the size of the global Koopman (computed above), ${\cal K}$ is $30\times 30$. Moreover the sparse structure of both of these Koopman operators can be seen in Fig. \ref{fig:bistable_Koopman_structures}. 
\begin{figure}[h!]
\begin{center}
\includegraphics[width=.6 \textwidth]{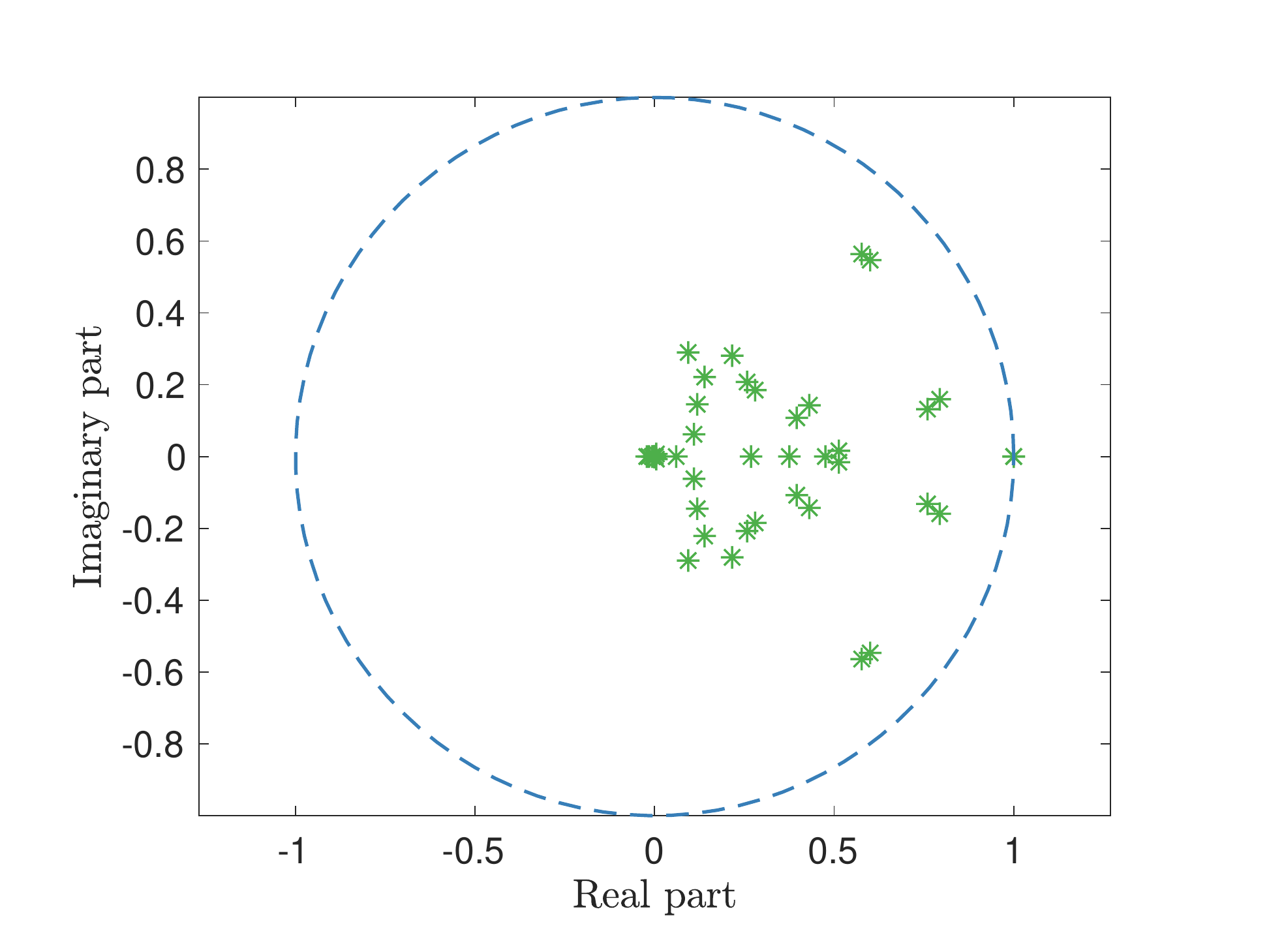}
\caption{Eigenvalues of the stitched Koopman operator ${\cal K}_{\cal S}$ corresponding to the bistable toggle switch system given in Eq. \eqref{eq:bistable_toggle_switch}.}
\label{fig:bistable_eigenvalues_stitched}
\end{center}
\end{figure}
%
\begin{figure}[h!]
\begin{center}
\subfigure[]{\includegraphics[width = 0.45 \linewidth]{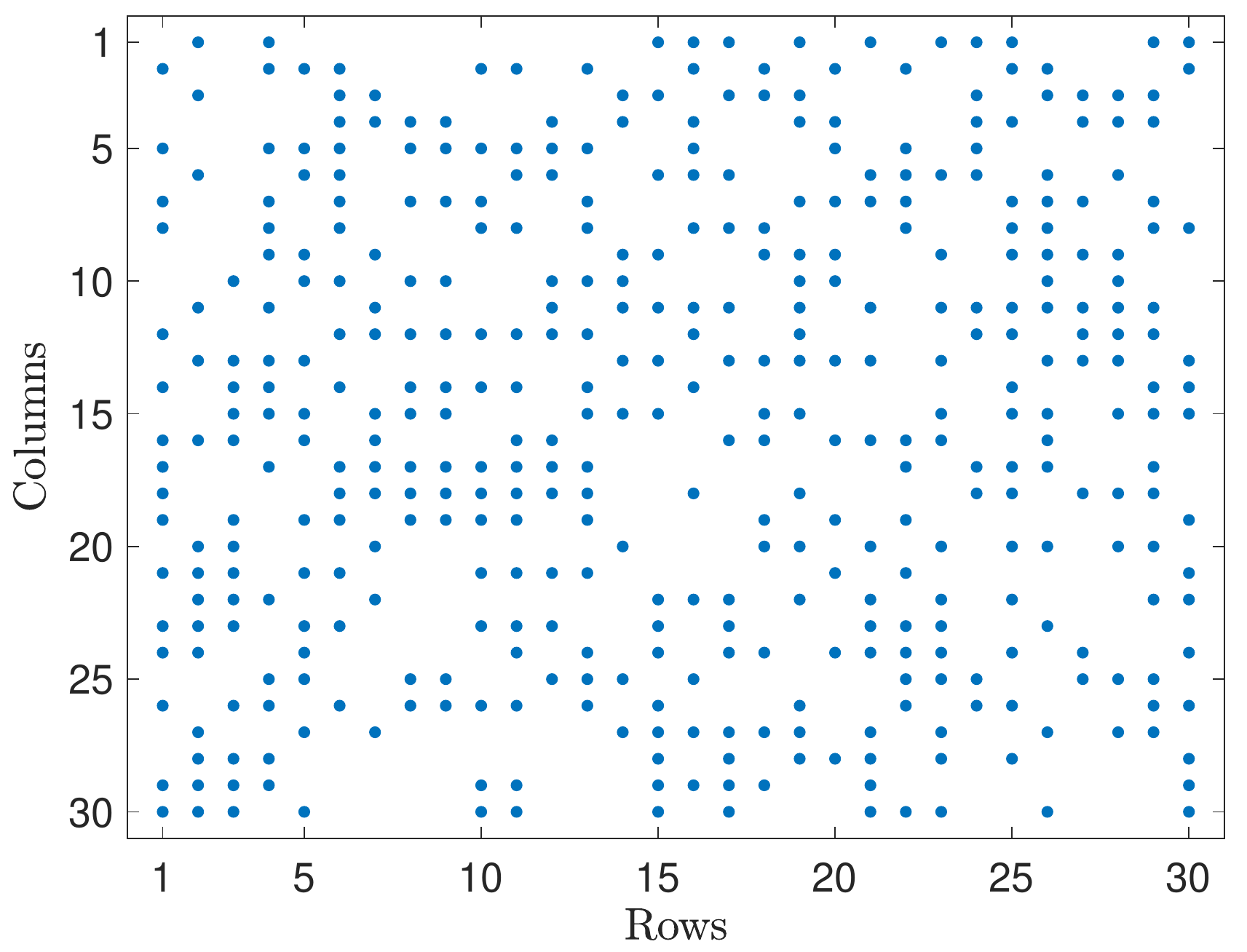}}
\subfigure[]{\includegraphics[width = 0.45 \linewidth]{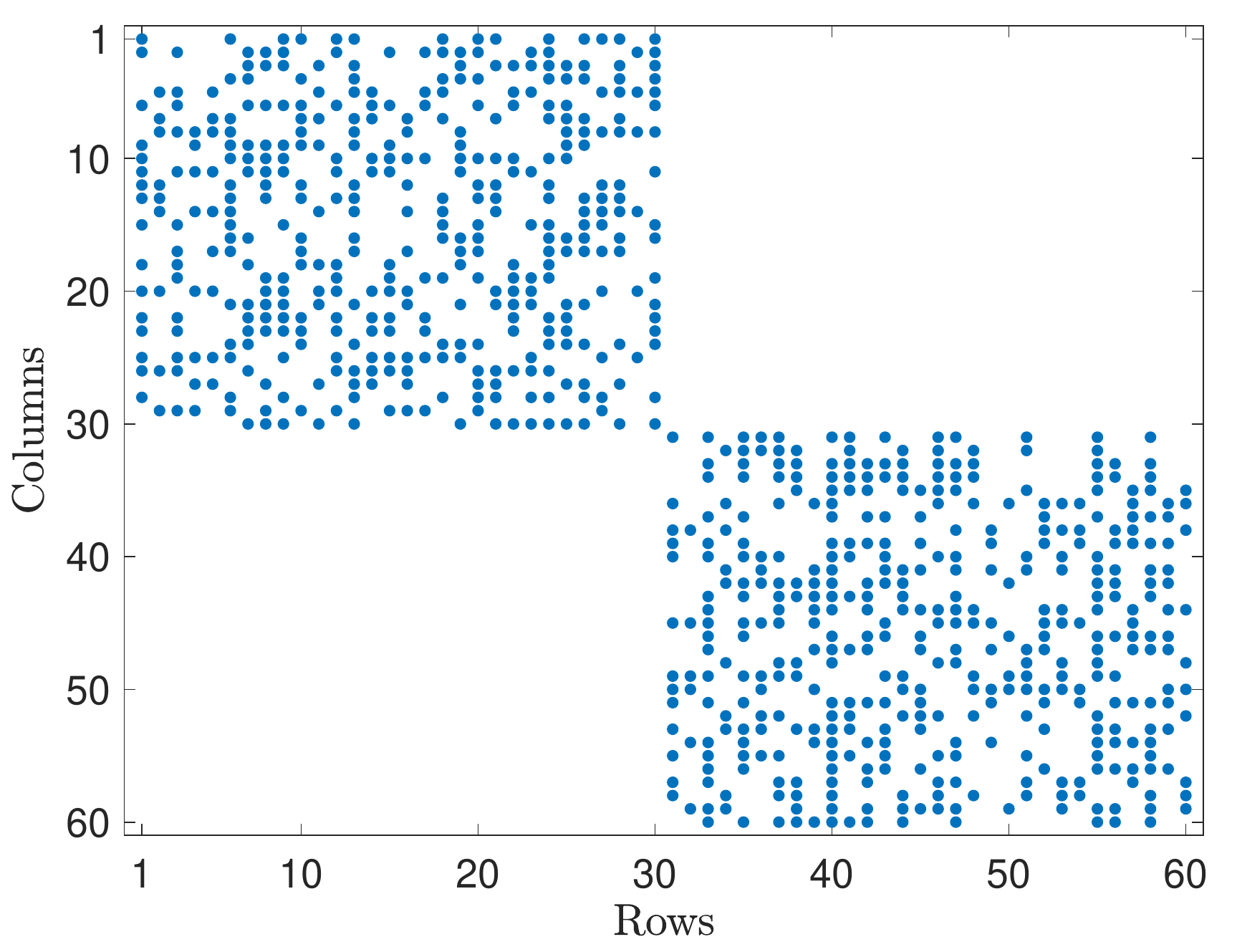}}
\caption{(a) Sparse structure of the Koopman operator ${\cal K}$ (b) Sparse structure of the stitched Koopman operator ${\cal K}_{\cal S}$.}
\label{fig:bistable_Koopman_structures}
\end{center}
\end{figure}

\textbf{Stitched Koopman Operator Validation:}
It is crucial to validate the stitched Koopman operator, ${\cal K}_{\cal S}$ if it captures the attractor sets on the state space of the dynamical system. To validate this, the eigenvalues are computed and their  eigenfunctions corresponding to the dominant eigenvalues are shown in Fig.  \ref{fig:bistable_inv_spaces_stitched}. Clearly Fig. \ref{fig:bistable_inv_spaces_stitched} shows that each of the attractor sets on the state space are captured by the eigenfunctions associated with the first two leading eigenvalues of ${\cal K}_{\cal S}$. 
\begin{figure}[h!]
\begin{center}
\subfigure[]{\includegraphics[width = 0.45 \linewidth]{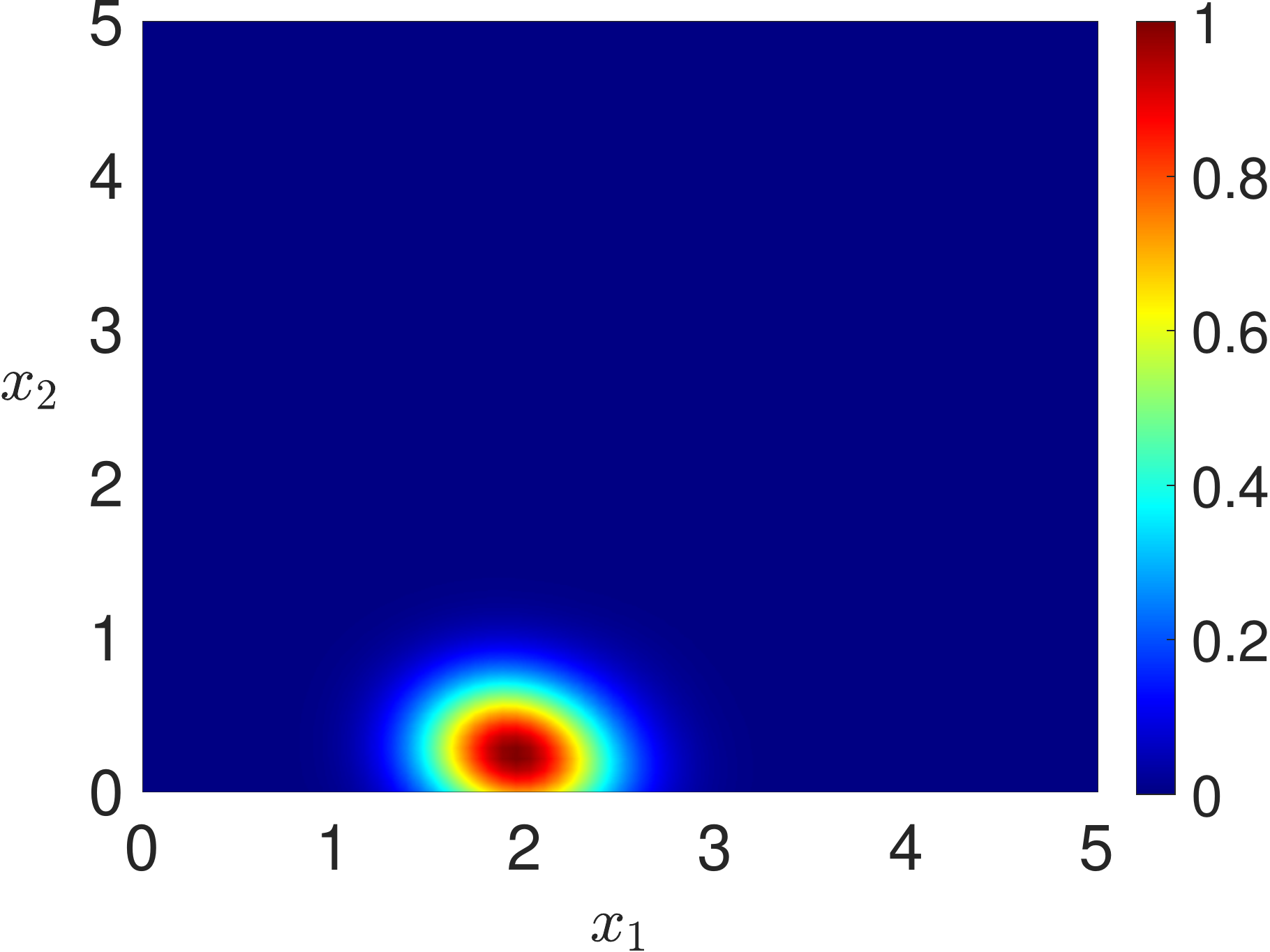}}
\subfigure[]{\includegraphics[width = 0.45 \linewidth]{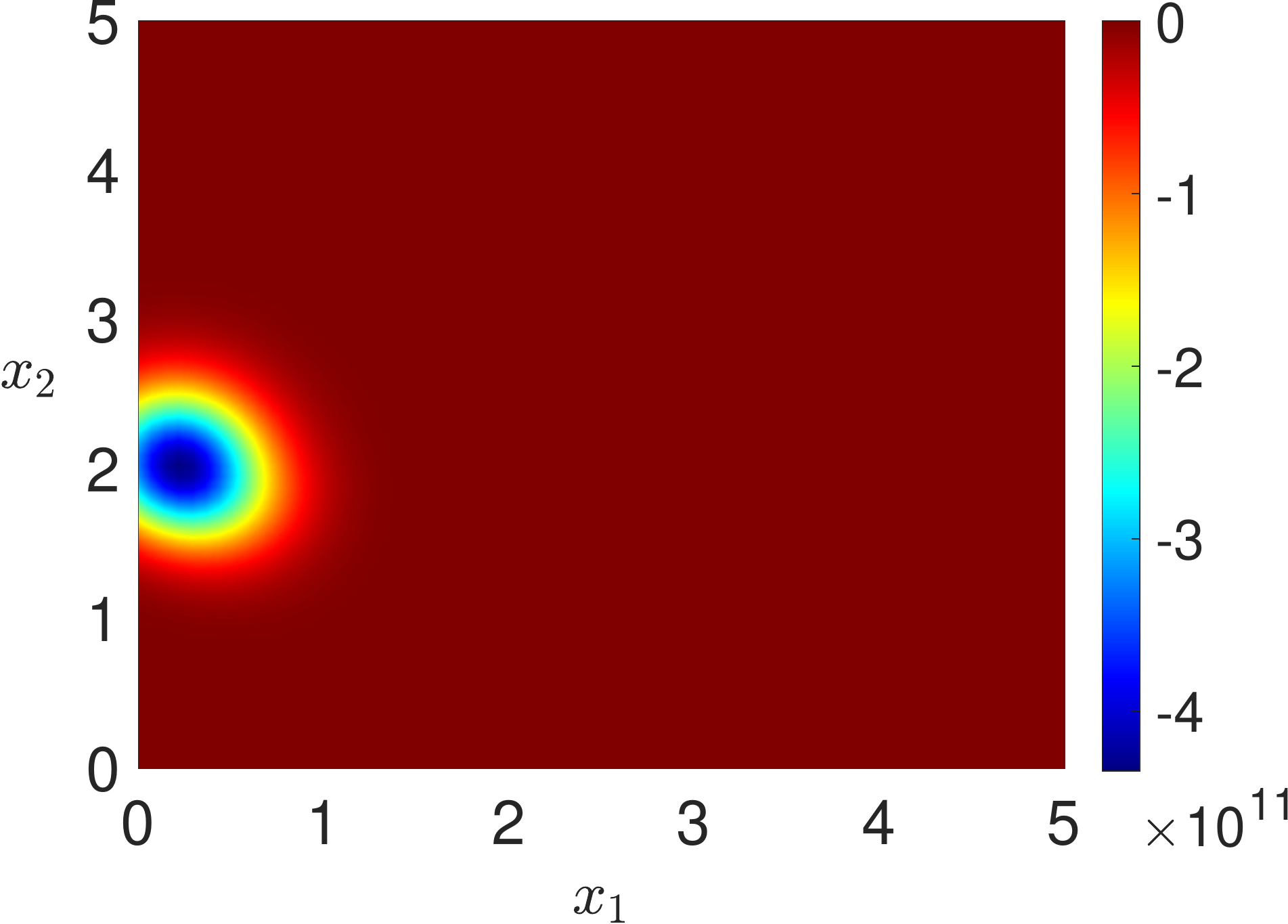}}
\caption{(a) and (b) Eigenvectors of ${\cal K}_{\cal S}$ associated with dominant eigenvalues $\lambda = 1$ on the state space. Eigenvectors of the stitched Koopman operator captures both the invariant sets of the state space.}
\label{fig:bistable_inv_spaces_stitched}
\end{center}
\end{figure}

Observe that, we have now computed two global Koopman operators, one assuming the knowledge of the entire state space to obtain ${\cal K}$ and the other by stitching the local invariant subspaces to obtain ${\cal K}_{\cal S}$. Fig. \ref{fig:bistable_inv_spaces_stitched} shows the attractor sets on the state space identified using ${\cal K}_{\cal S}$ and moreover they also approximate well the invariant sets shown in Fig. \ref{fig:bistable_inv_spaces}. This validates the proposed approach of phase space stitching to compute the global Koopman operator from local Koopman operators. 

\textbf{Global Koopman Operator using Symmetry Properties:}
Clearly, the phase space of the bistable toggle switch consists of two invariant subspaces, $x_1 > x_2$ and $x_1 < x_2$ respectively and they are symmetric to each other. In particular, these two invariant subspaces are reflective of each other. In the scenario where the knowledge of this symmetry is known and the time-series from an experiment is available on any of the invariant subspaces, then phase space stitching (Algorithm \ref{alg}) is invoked to identify the stitched Koopman or equivalently global Koopman operator.  

We demonstrate the global Koopman operator computation using DMD, however additional work is needed to show this under EDMD or deepDMD. The local Koopman operator corresponding to the region $x_1 > x_2$ using the DMD is given by
\begin{align*}
{\cal K}_{right} = & \begin{bmatrix} 0.6039 & 0.0313 \\ -0.4784 & 1.0375 \end{bmatrix}.
\end{align*}
Then the Koopman operator corresponding to the reflective space $x_1 < x_2$ is obtained as 
\begin{align*}
    {\cal K}_{left} = & \gamma^{-1}{\cal K}_{right}\gamma=\begin{bmatrix}
    1.0375 & -0.4784 \\ 0.0313 & 0.6039
    \end{bmatrix}. 
\end{align*}
where $\gamma$ denotes the reflective transformation between the invariant subspaces such that $(x_1,x_2)^\top\xmapsto{\gamma} (x_2,x_1)^\top$. The matrix representation for $\gamma$ corresponding to the reflective symmetry is given by $\begin{pmatrix} 0 & 1 \\ 1 & 0 \end{pmatrix}$. 
Then the global stitched Koopman operator is given by 
\begin{align*}
    {\cal K}_{global} = \begin{bmatrix}
    {\cal K}_{right} & \\ & {\cal K}_{left}
    \end{bmatrix}.
\end{align*}

We next demonstrate the phase space stitching results on a nonlinear second-order system. 

\subsection{Second Order System with Bilinear and Quadratic Terms}
Consider a second-order dynamical system governed by the dynamics:
\begin{equation}
\begin{aligned}
\dot{x}_1 = & x_1 - x_1 x_2 \\
\dot{x}_2 = & x_1^2 -2x_2
\end{aligned}
\label{eq:heart_dyn_system}
\end{equation}
The system \eqref{eq:heart_dyn_system} has 3 equilibrium points at $(\sqrt{2},1)$, $(-\sqrt{2},1)$ and $(0,0)$. It is seen that the origin is a saddle point and the other two equilibrium points are stable. 

\textbf{Data Generation:} 
The Koopman operator is obtained by training on the complete state space data with $81$ initial conditions and for each initial condition, $21$ time-points and is denoted by ${\cal K}$ where $30$ Gaussian radial basis functions with $\sigma = 0.4$ are used. The observable functions used for the bistable toggle switch are used for this system as well however the centers for the radial basis functions are computed according to the time-series data of \eqref{eq:heart_dyn_system}.  The phase portrait of this system is shown in Fig. \ref{fig:phase_portrait_heart}. 
\begin{figure}[h!]
\begin{center}
\includegraphics[width = 0.6 \textwidth]{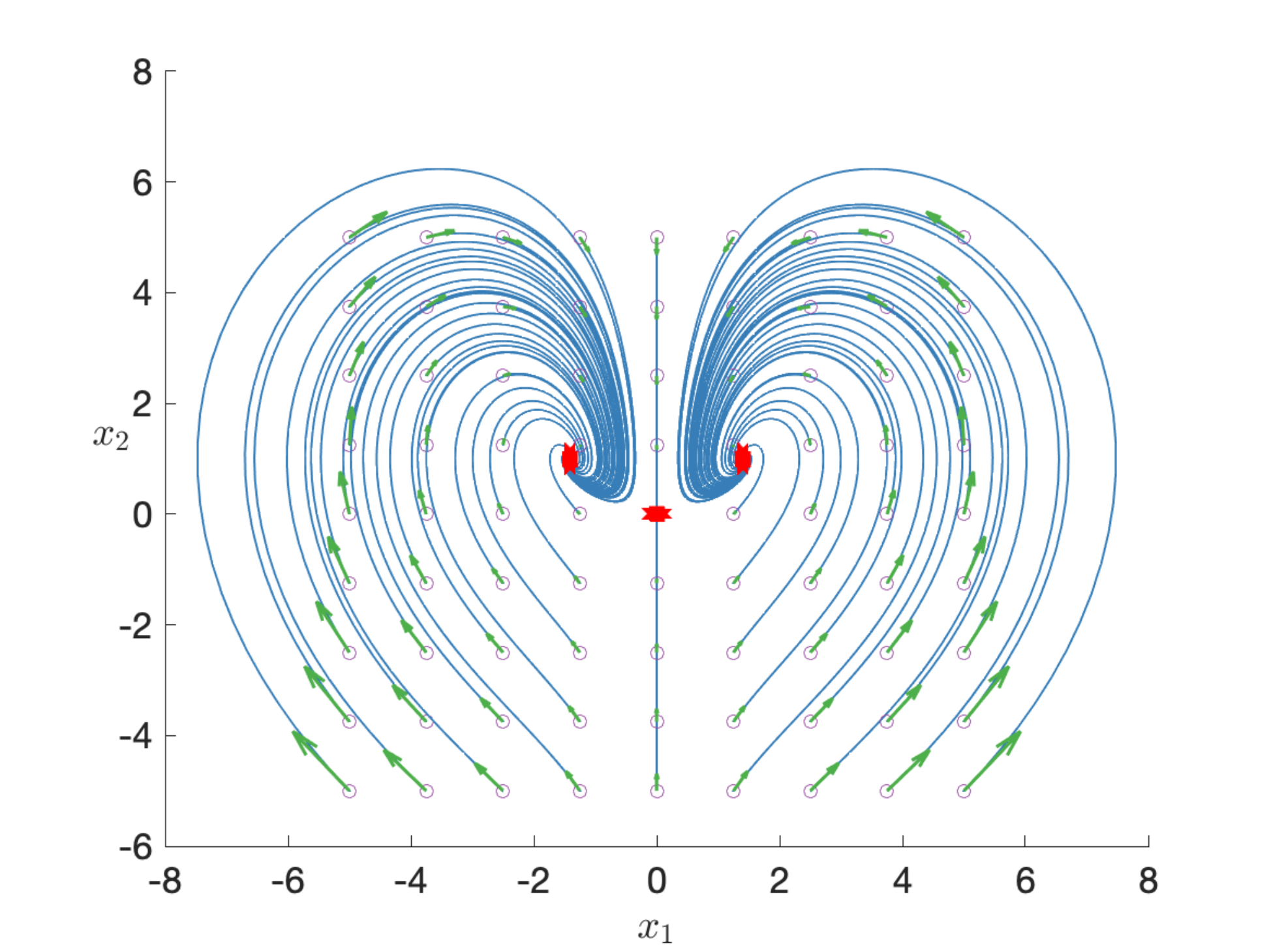}
\caption{Phase portrait corresponding to the system \eqref{eq:heart_dyn_system}. Circles indicate the initial conditions, green arrow indicate the direction of the vector field at the initial condition and the red dots are the equilibrium points.}
\label{fig:phase_portrait_heart}
\end{center}
\end{figure}

\textbf{Phase Space Exploration (Global to Local):} 
The global Koopman operator, ${\cal K}$ is computed and the eigenvalues of ${\cal K}$ are shown in Fig. \ref{fig:eigenvalues_heart}. Clearly the three dominant eigenvalues of ${\cal K}$ are located at $\lambda = 1$ with geometric multiplicity ($m_g$) equal to 1. The eigenfunctions corresponding to these dominant eigenvalues of ${\cal K}$ are shown in Fig. \ref{fig:inv_spaces_heart}.
\begin{figure}[h!]
\begin{center}
\includegraphics[width = 0.6 \textwidth]{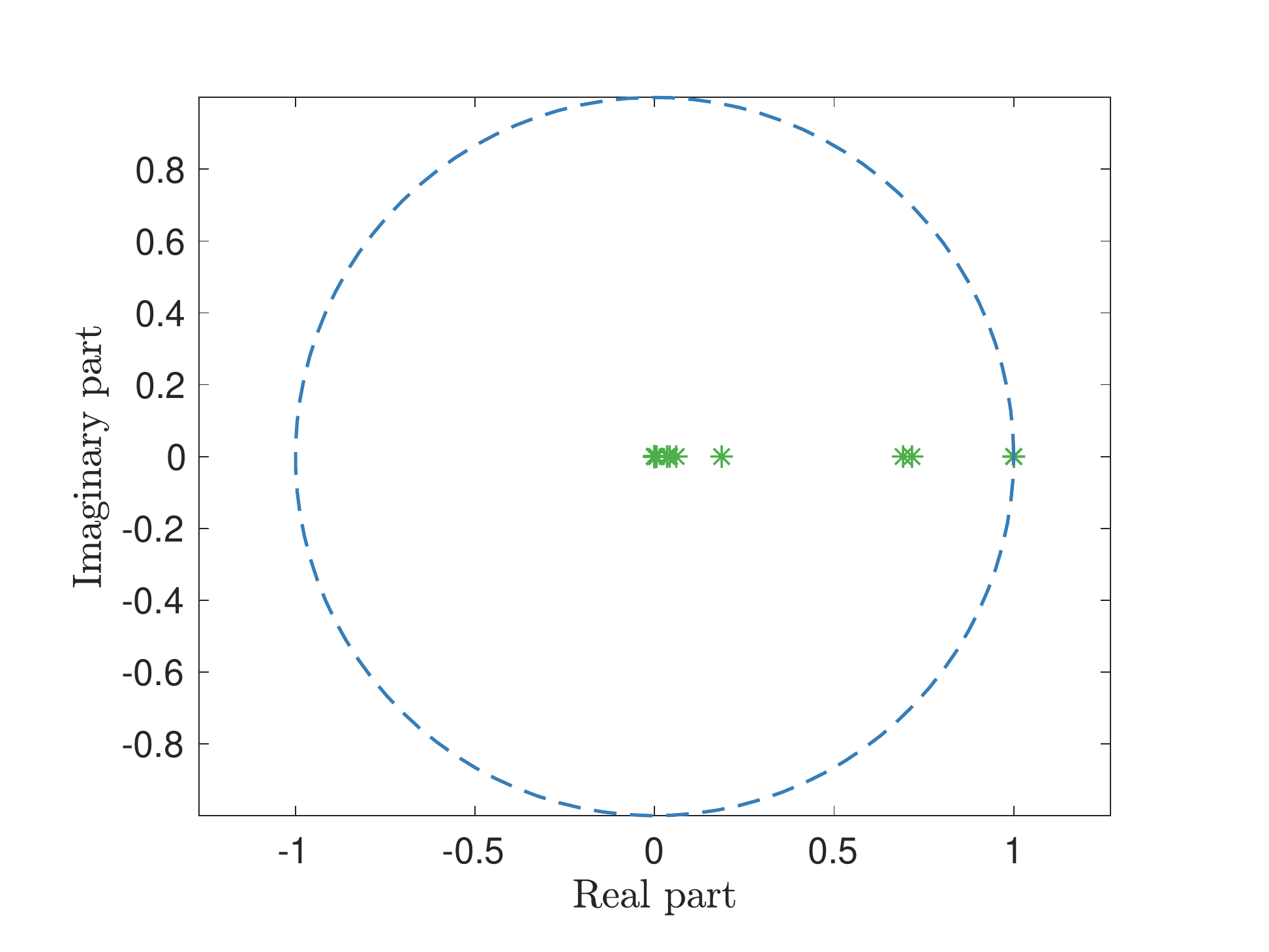}
\caption{Eigenvalues of the stitched Koopman operator, ${\cal K}$ corresponding to the system given in Eq. \eqref{eq:heart_dyn_system}.}
\label{fig:eigenvalues_heart}
\end{center}
\end{figure}
\begin{figure}[h!]
    \centering
    \subfigure[]{\includegraphics[width = 0.32 \textwidth]{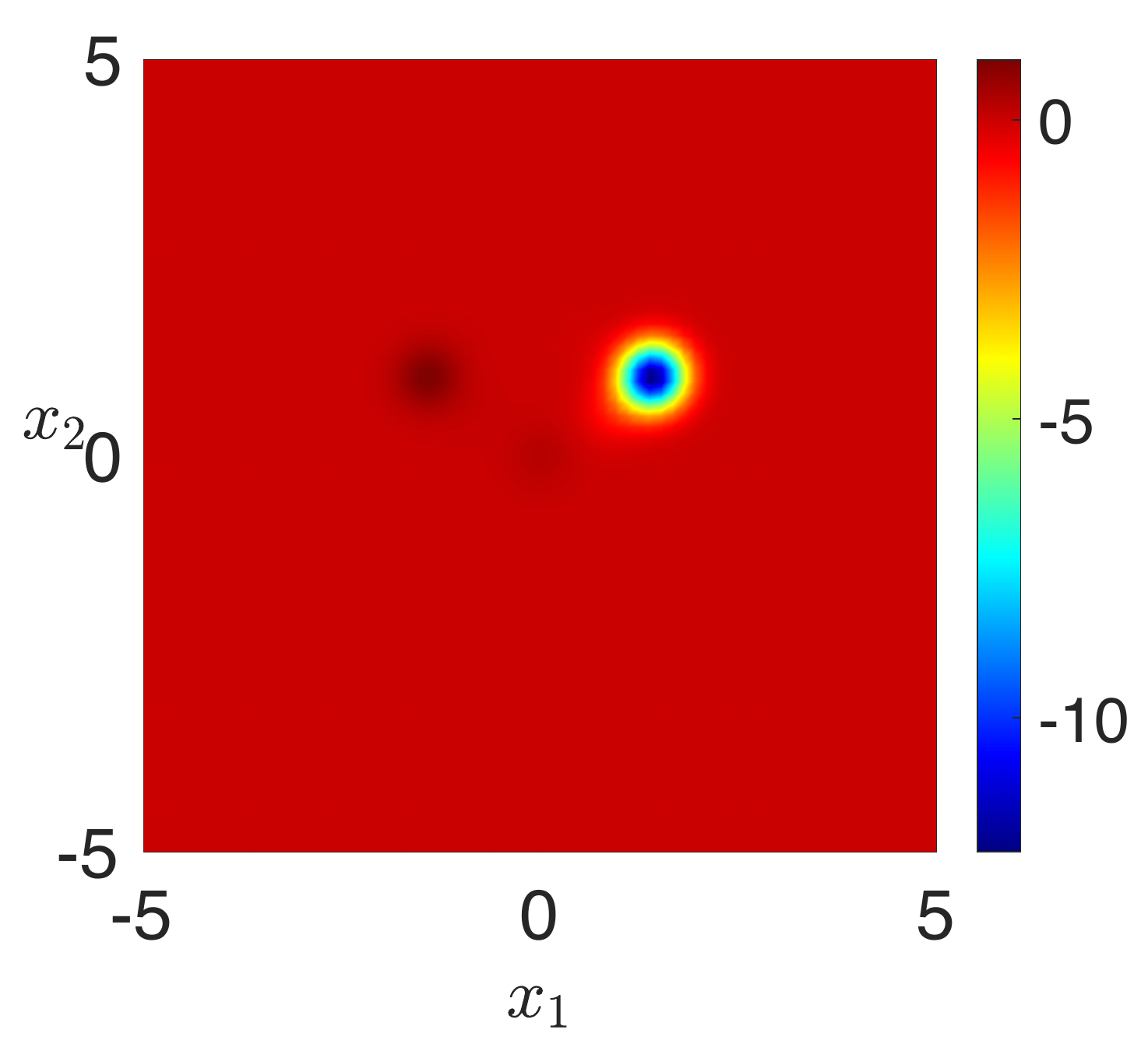}}
    \subfigure[]{\includegraphics[width = 0.32 \textwidth]{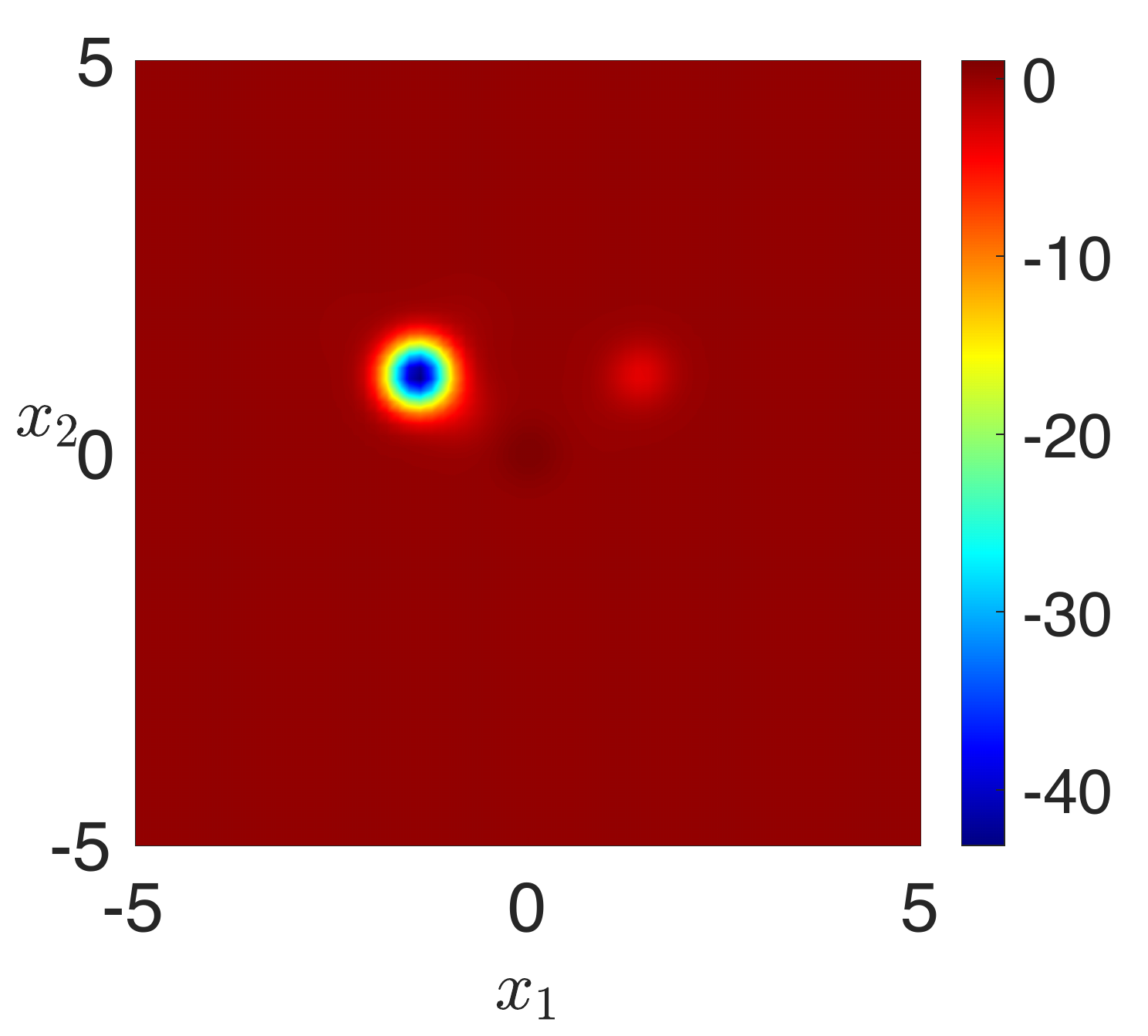}}
    \subfigure[]{\includegraphics[width = 0.32 \textwidth]{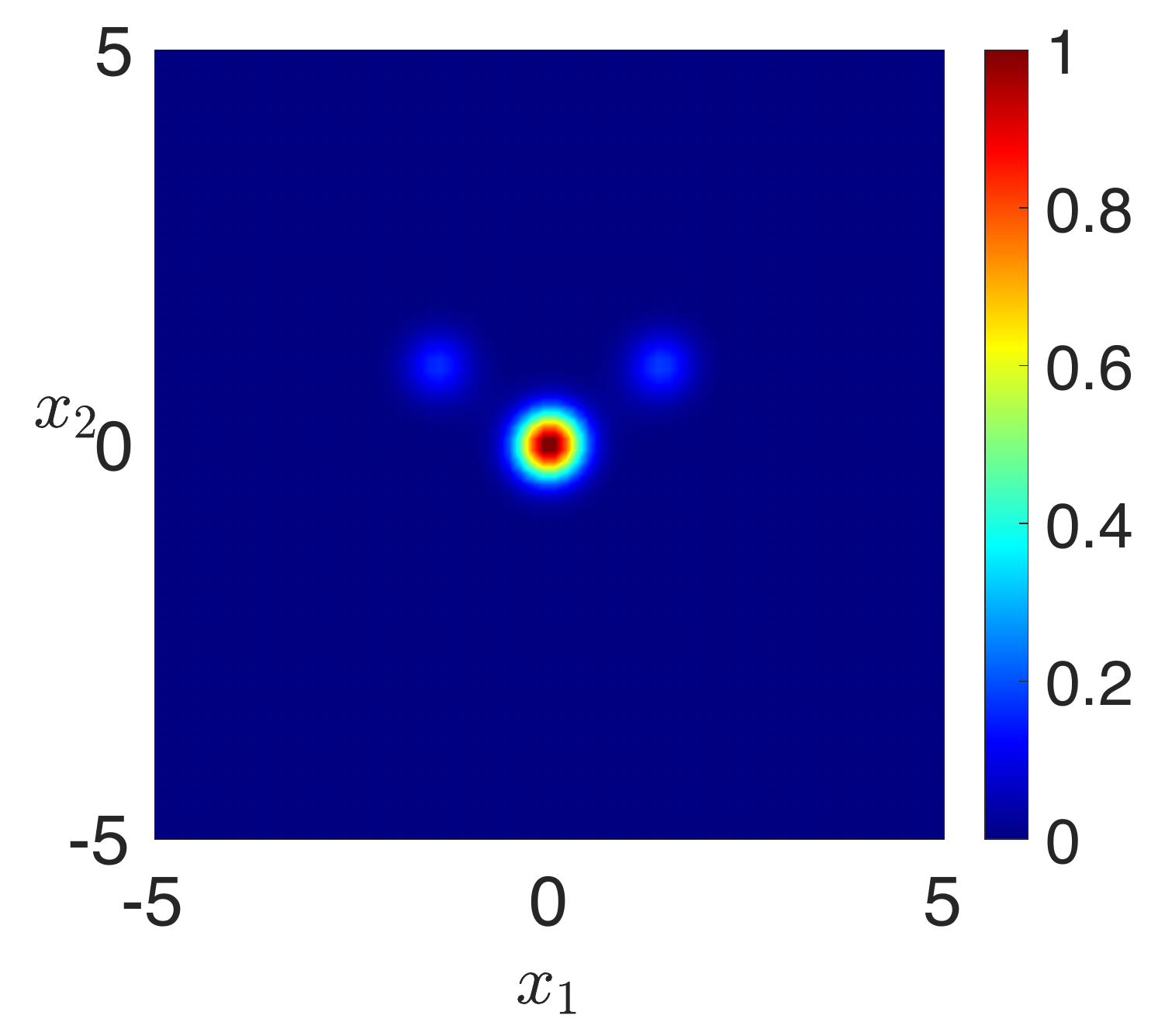}}
    \caption{Eigenvectors corresponding to the dominant eigenvalues of ${\cal K}$ on state space. (a) and (b) stable equilibrium points and (c) saddle point. }
    \label{fig:inv_spaces_heart}
\end{figure}

\textbf{Phase Space Exploration (Local to Global):}
Assume there is no knowledge of the complete phase space and the knowledge about the local invariant sets are only known. Then, the corresponding local Koopman operators are denoted by ${\cal K}_{left}$ and ${\cal K}_{right}$. Furthermore, the stitched Koopman operator with respect to these local Koopman operators is given by ${\cal K}_{\cal S}$. 
The eigenvalues and the sparse structure of the stitched Koopman operator when compared to ${\cal K}$ are shown in Fig. \ref{fig:eigenvalues_heart_stitched} and Fig. \ref{fig:heart_Koopman_structures} respectively. 
\begin{figure}[h!]
\begin{center}
\includegraphics[width = 0.6 \linewidth]{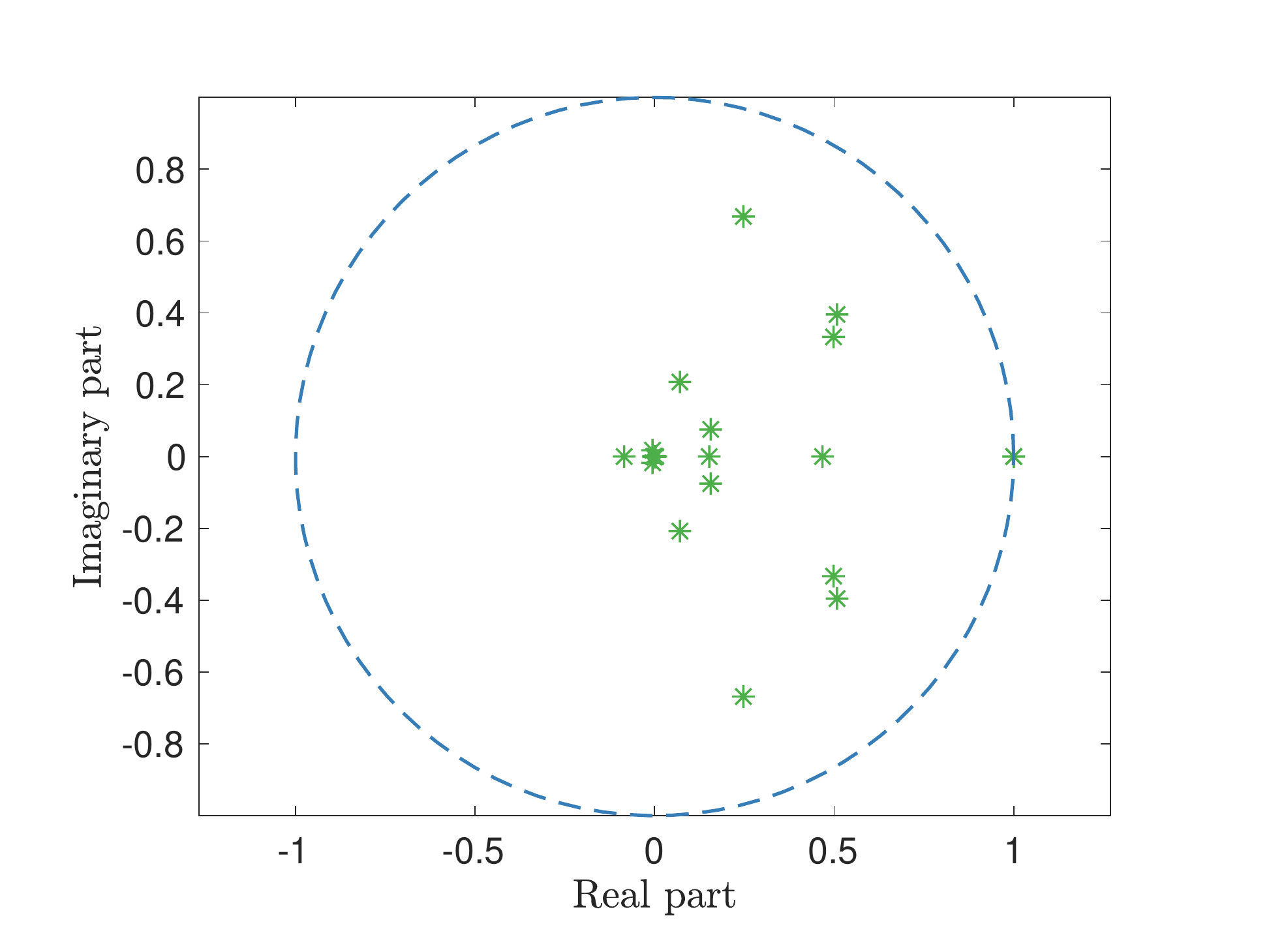}
\caption{Eigenvalues of the stitched Koopman operator ${\cal K}_{\cal S}$ corresponding to the system \eqref{eq:heart_dyn_system}.}
\label{fig:eigenvalues_heart_stitched}
\end{center}
\end{figure}
\begin{figure}[h!]
\begin{center}
\subfigure[]{\includegraphics[width = 0.45 \linewidth]{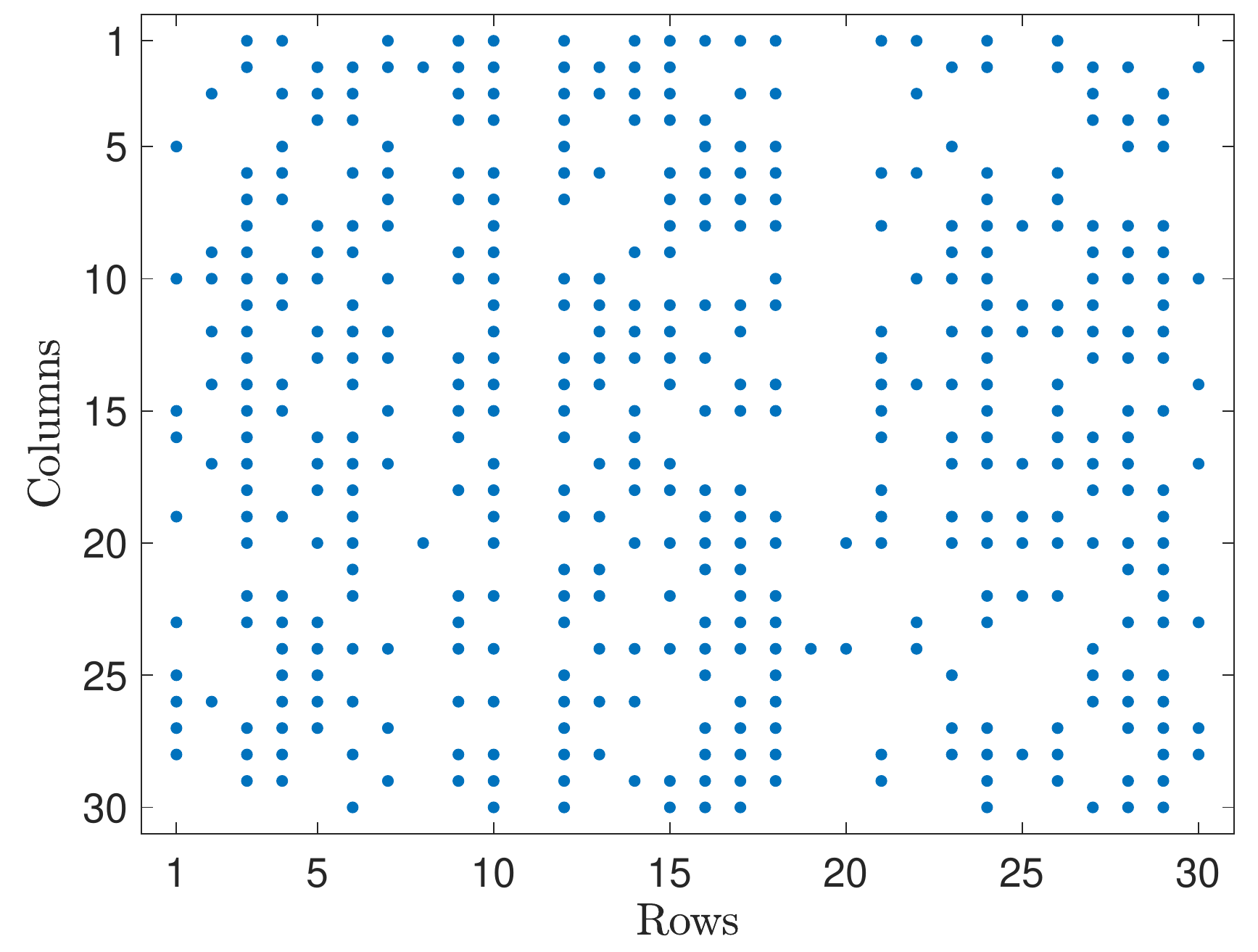}}
\subfigure[]{\includegraphics[width = 0.45 \linewidth]{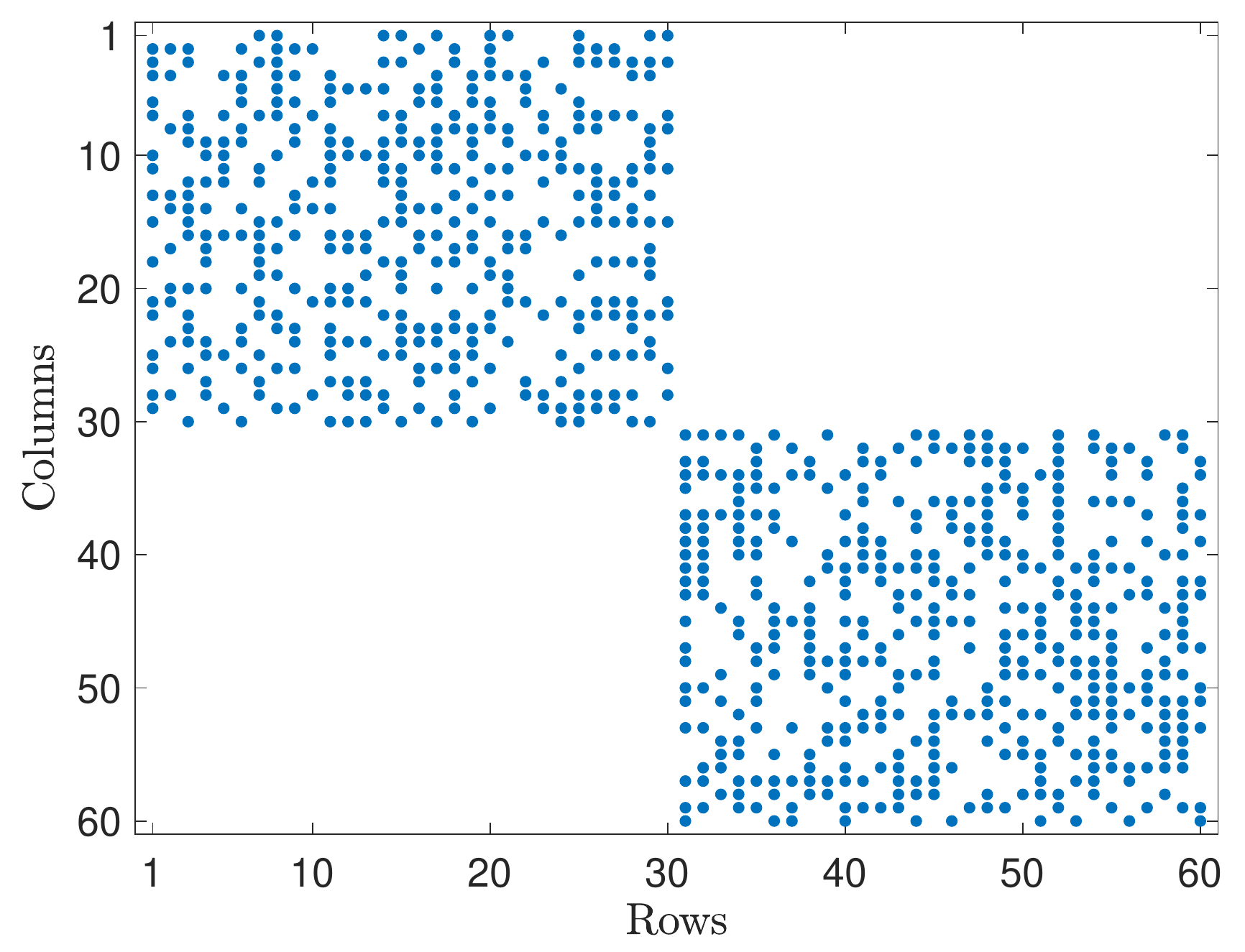}}
\caption{(a) Sparse structure of the Koopman operator ${\cal K}$ (b) Sparse structure of the stitched Koopman operator ${\cal K}_{\cal S}$.}
\label{fig:heart_Koopman_structures}
\end{center}
\end{figure}

\textbf{Stitched Koopman Operator Validation:}
The eigenvector plots corresponding to the dominant eigenvalues of ${\cal K}_{\cal S}$ are shown in Fig. \ref{fig:heart_inv_spaces_stitched} (a) and (b). It can be seen that the stitched Koopman operator identifies the attractor sets in the state space. 
\begin{figure}[h!]
\begin{center}
\subfigure[]{\includegraphics[width = 0.45 \linewidth]{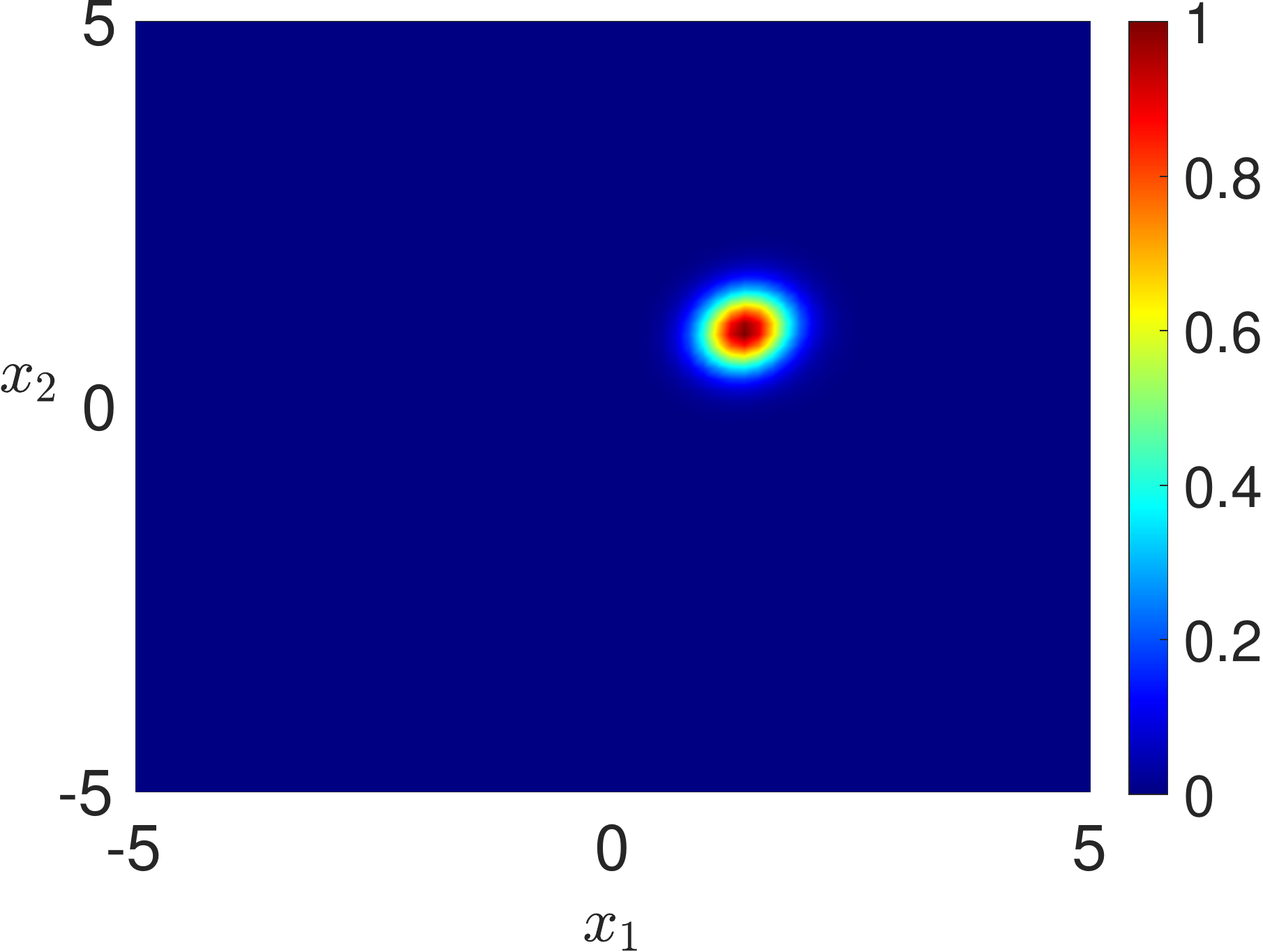}}
\subfigure[]{\includegraphics[width = 0.45 \linewidth]{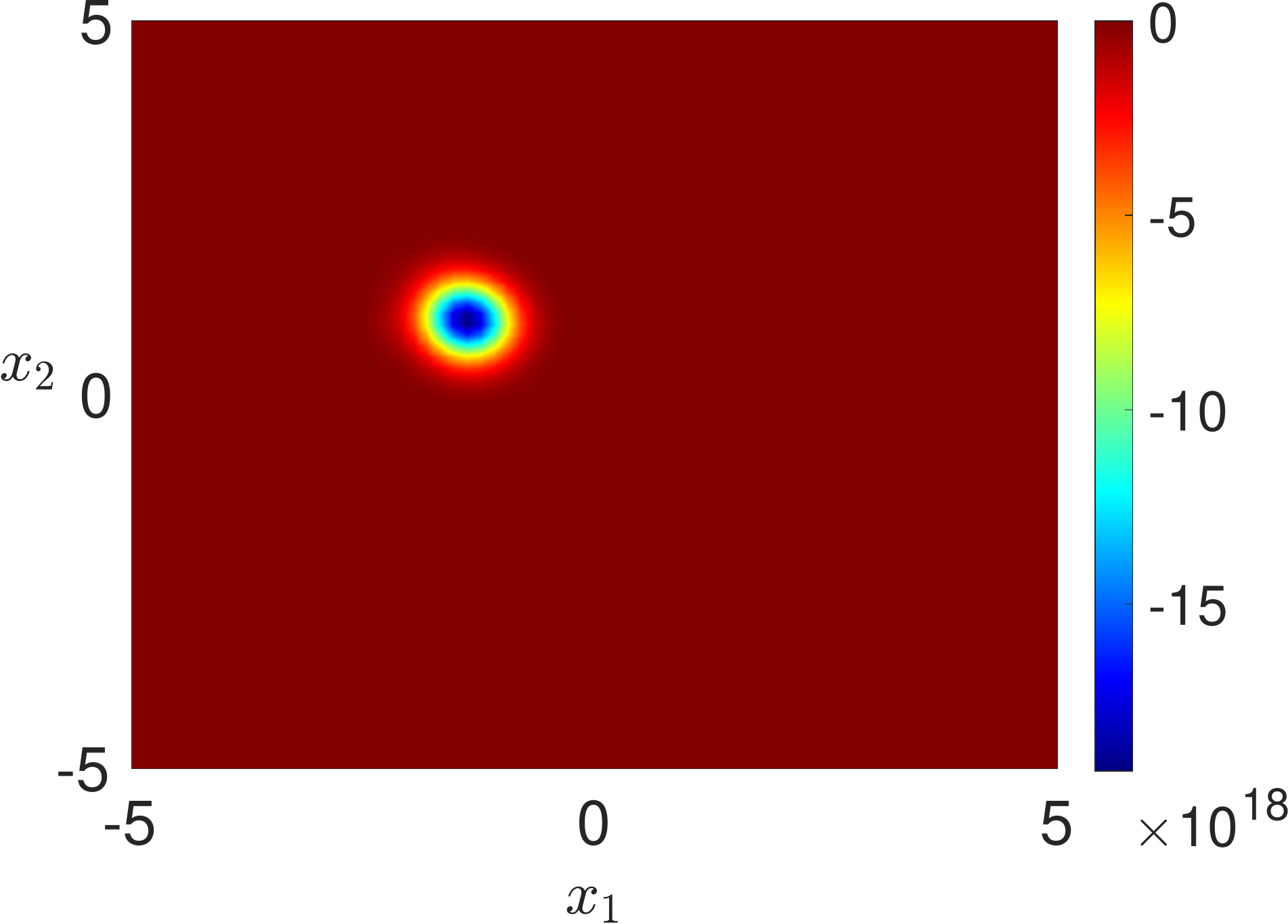}}
\caption{(a) and (b) Eigenvectors of ${\cal K}_{\cal S}$ associated with dominant eigenvalues $\lambda = 1$ on the state space. Eigenvectors of the stitched Koopman operator captures both the invariant sets of the state space.}
\label{fig:heart_inv_spaces_stitched}
\end{center}
\end{figure}
Therefore, the stitched Koopman operator describes the global behavior of the system and identifies the attractor sets of the phase space in comparison to ${\cal K}$ (which me may not have access to in real experiments).

\textbf{Global Koopman Operator using Symmetry Properties:} 
We now discuss how the global Koopman operator can be identified if the knowledge of time-series data from any one invariant subspace is available and under the assumption that the type of symmetry between the invariant subspaces of the dynamical system is known. We demonstrate the global Koopman operator computation using DMD, however additional work is needed to show this under EDMD or deepDMD. Let's say we have access to the time-series data in the region $M_{right} = \{x \in \mathbb{R}^2: x_1 > 0\}$. The local Koopman operator applying DMD on the data is given by 
\begin{align*}
    {\cal K}_{right} = \begin{bmatrix}     0.9782  &  0.0253 \\
    0.7755 &  -0.0955
    \end{bmatrix}.
\end{align*}
Let $M_{left}$ be another invariant subspace such that, $M_{left} = \{x \in \mathbb{R}^2: x_1 < 0\}$. Under the assumption that the invariant subspaces $M_{right}$ and $M_{left}$ has a reflective symmetry, that is, $(x_1,x_2)^{\top}\xmapsto{\gamma} (-x_1,x_2)^{\top}$. Then it follows from Theorem \ref{K_i_K_j_theorem} that the Koopman operator corresponding to $M_{left}$ is given by
\begin{align*}
    {\cal K}_{left} = \gamma^{-1} {\cal K}_{right} \gamma = \begin{bmatrix}    0.9782 &   -0.0253 \\ 
   -0.7755 &  -0.0955
    \end{bmatrix},
\end{align*}
where $\gamma = \begin{pmatrix}-1 & 0 \\ 0 & 1\end{pmatrix}$ is the 2-dimensional representation of the non-identity element in the symmetry group. Moreover, the local Koopman on $M_{left}$ matches with the local Koopman computed from time-series data of $M_{left}$. Therefore applying the phase space stitching result (from Section \ref{sec:phase_space_stitching}) by noticing that the observables are indeed the states, the global Koopman is given by 
\begin{align*}
    {\cal K}_{global} = \begin{bmatrix} {\cal K}_{right} & 0 \\ 0 & {\cal K}_{left} \end{bmatrix}.
\end{align*}

\subsection{Topologically Conjugate Systems} 

In the following study, we consider two topologically conjugate dynamical systems and with the knowledge of the topological conjugacy and one of the systems time-series data, the evolution of the other system is identified applying the results developed in Section \ref{sec:TC_systems}. 
Let $T_1: \mathbb{R}^2 \to \mathbb{R}^2$ and $T_2:\mathbb{R}^2 \to \mathbb{R}^2$ be two dynamical systems that are topologically conjugate with the homeomorphism $h:\mathbb{R}^2 \to \mathbb{R}^2$ such that 
\begin{align*}
    \dot{x} = & T_1(x) \\
    \dot{y} = & T_2(y)
\end{align*}
where
\begin{align*}
    x = \begin{bmatrix} x_1 \\ x_2 \end{bmatrix}, \quad 
    y = \begin{bmatrix} y_1 \\ y_2 \end{bmatrix}, \quad 
    T_1(x) = \begin{bmatrix} -x_1 \\ -x_2 +x_1^2 \end{bmatrix}, \quad 
    T_2(y) = \begin{bmatrix} -y_1 \\
    -y_2 \end{bmatrix}.
\end{align*}

The phase portraits for the system $T_1$ and $T_2$ are shown in Fig. \ref{fig:phase_portraits_TC} (a) and (b) respectively. 
\begin{figure}[h!]
\centering
\subfigure[]{\includegraphics[width = 0.45 \linewidth]{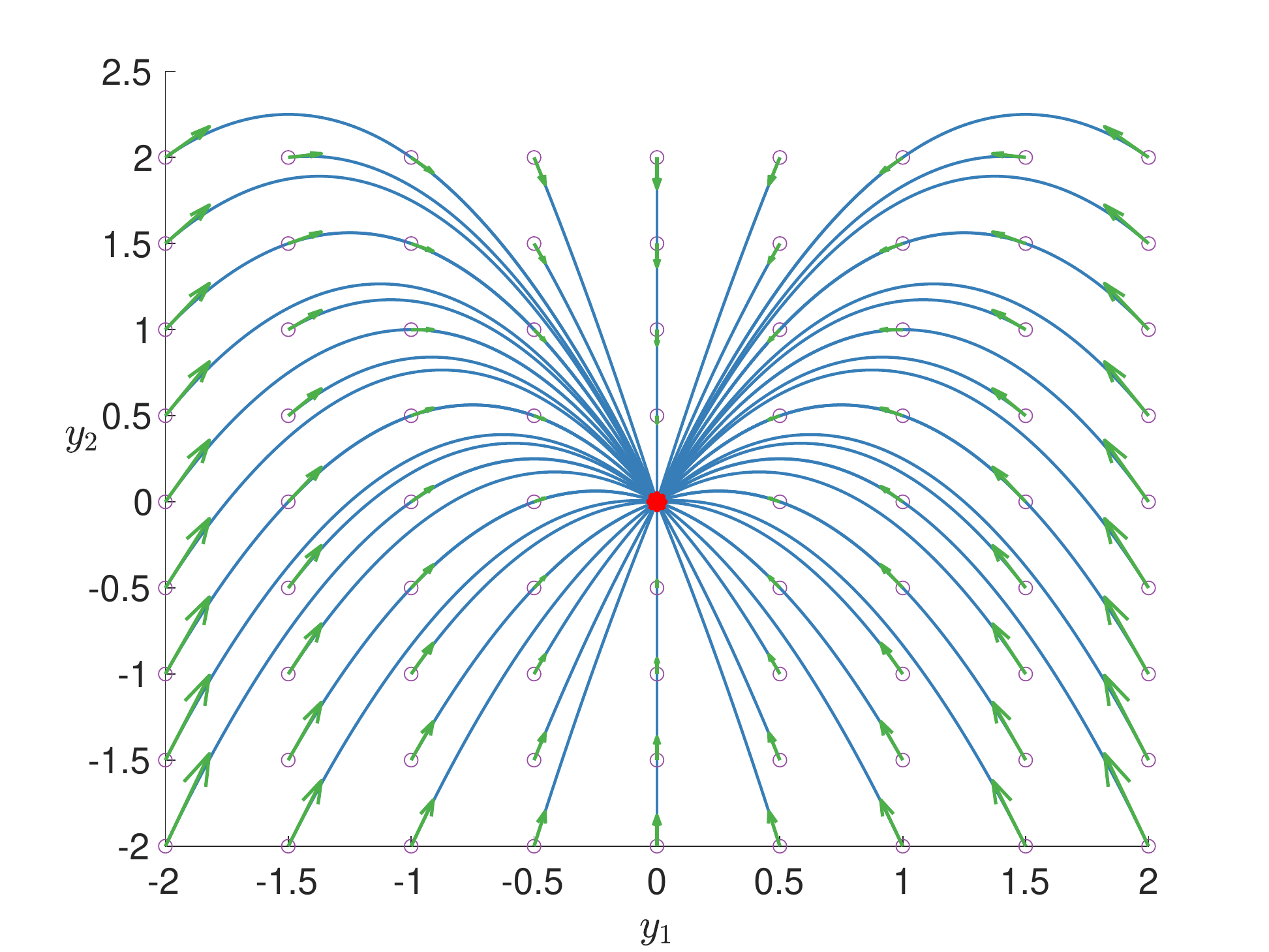}}
\subfigure[]{\includegraphics[width = 0.45 \linewidth]{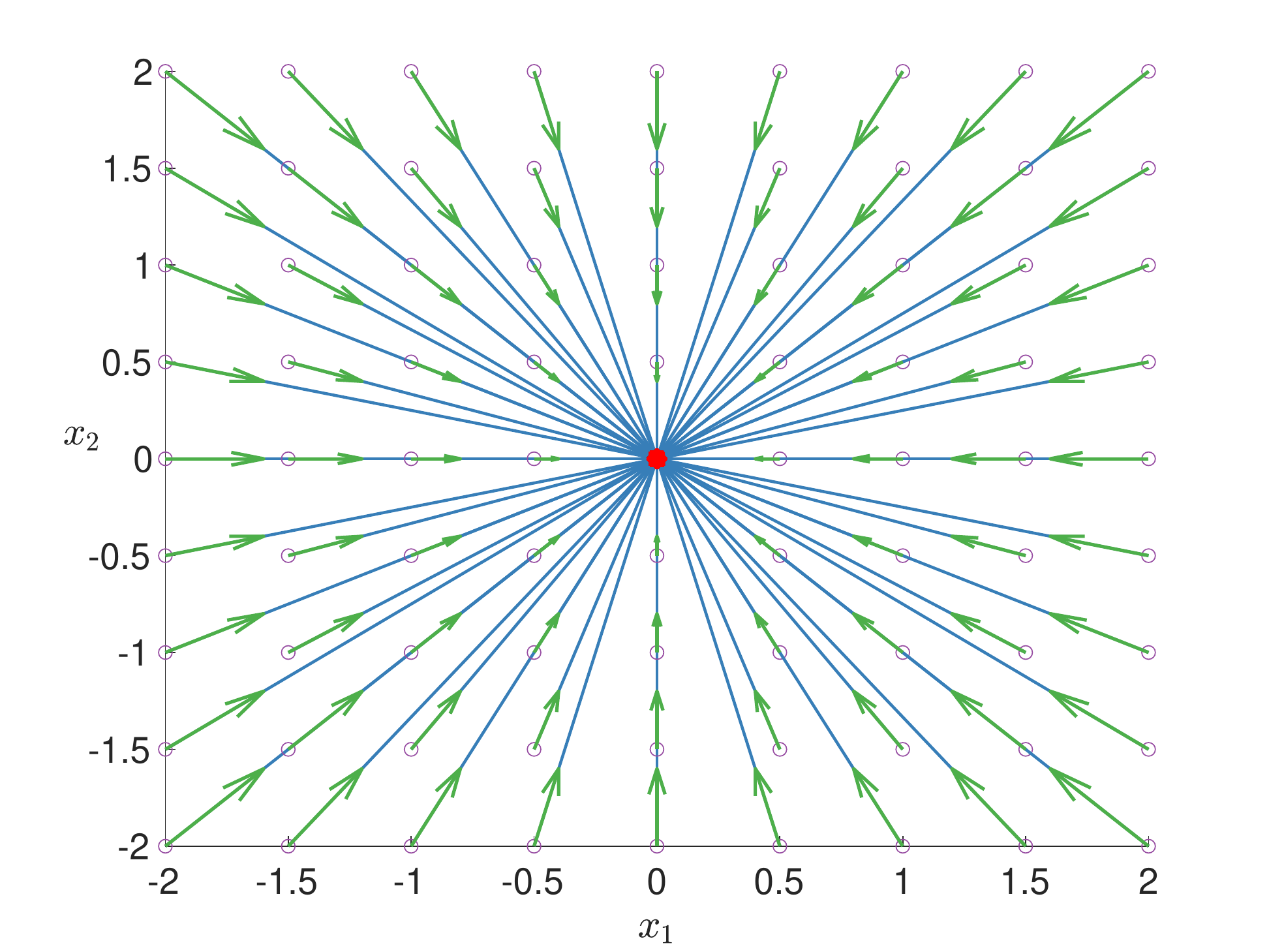}}
\caption{(a) Phase portrait of $T_1$ and (b) Phase portrait of $T_2$. The dynamical systems $T_1$ and $T_2$ are topologically conjugate via the homeomorphism $h$ as shown in Eq. \eqref{eq:homeomorphic_function}.}
\label{fig:phase_portraits_TC}
\end{figure}
The homeomorphic function between $T_1$ and $T_2$ is given by 
\begin{align}
    h(y) = \begin{bmatrix} y_1 \\ y_2 - y_1^2 \end{bmatrix}
    \label{eq:homeomorphic_function}
\end{align}

To demonstrate the application of proposed Koopman operator theoretic methods for topologically conjugate dynamical systems in Section \ref{sec:TC_systems}, we begin with the assumption of knowledge of system $T_2$ and the homeomorphism $h$ and eventually study the dynamical system $T_1$. 

\textbf{Data Generation:} Time-series data corresponding to $81$ initial conditions and for each initial condition, $21$ time-points are collected for system $T_2$. This time-series data is used to compute the finite dimensional approximate Koopman operator, ${\cal K}_{\Theta}$ using EDMD by choosing the dictionary function as 
\begin{align}
    \Theta(y) = \begin{bmatrix} y_1 \\ y_2 \\ y_1^2 \end{bmatrix}
    \label{eq:obsvs_psi}
\end{align}
Then the finite dimensional Koopman with respect to the chosen observables, $\Theta$ is given by $ {\cal K}_{\Theta}$ as shown in Fig. \ref{fig:K_Theta_tc}. 
\begin{figure}[h!]
    \centering
    \includegraphics[scale = 0.3]{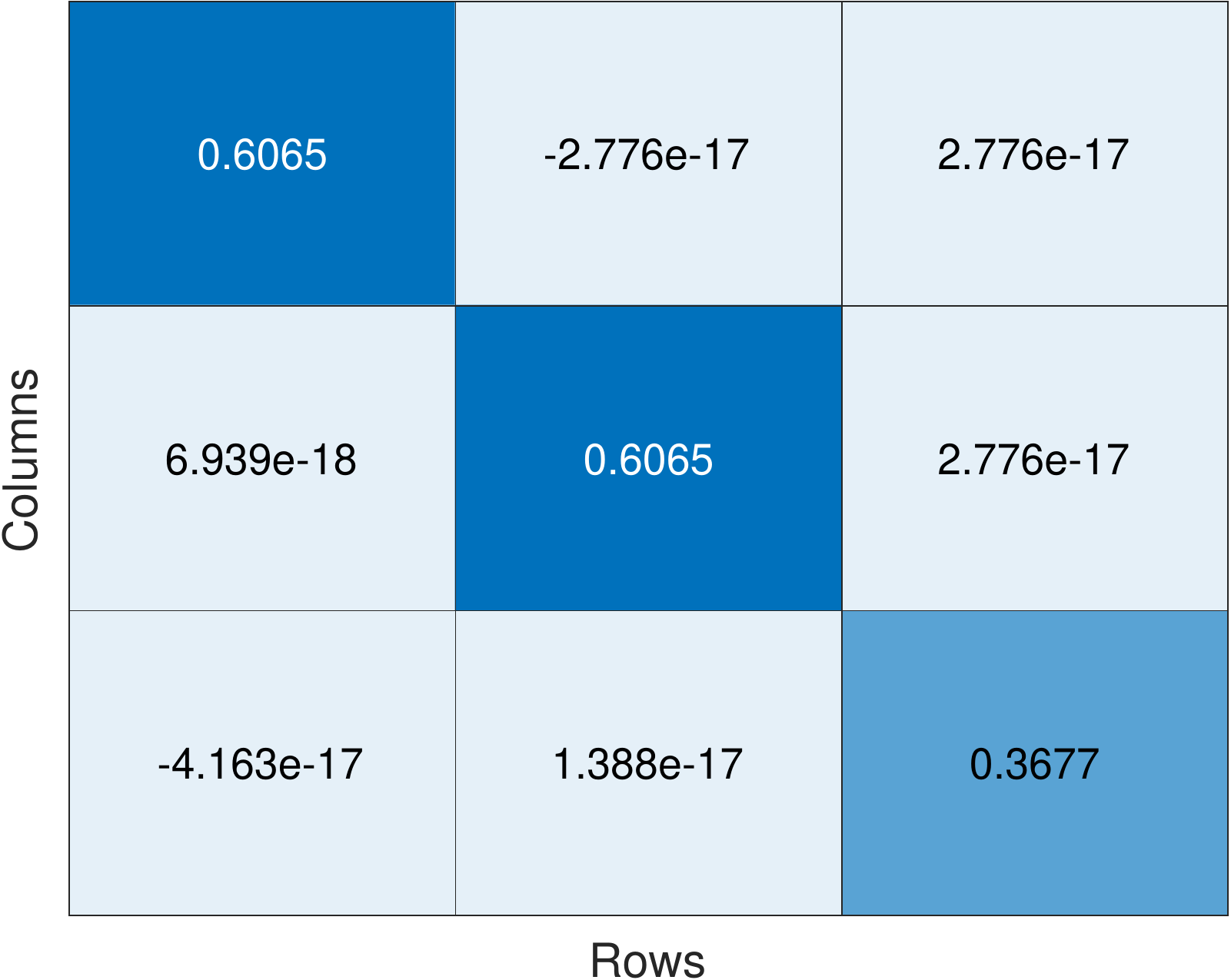}
    \caption{The matrix representation of the Koopman operator ${\cal K}_{\Theta}$ with the choice of observables $\Theta$ defined in Eq. \eqref{eq:obsvs_psi}.}
    \label{fig:K_Theta_tc}
\end{figure}

\textbf{Koopman Operator using Topological Conjugacy:} Once the dictionary functions for $T_2$ is fixed, the dictionary functions corresponding to the dynamical system $T_1$ is chosen as
$\Psi = \Theta \circ h^{-1}$, so that we have
\begin{align}
    \Psi(x) = \begin{bmatrix} x_1 \\ x_2+x_1^2 \\ x_1^2\end{bmatrix}
    \label{eq:obsvs_theta}
\end{align}

The time-series data corresponding to $T_1$ is obtained by the action of the homeomorphic function on the time-series data on $T_2$. Then the corresponding finite dimensional Koopman operator denoted by ${\cal K}_{\Psi}$ is shown in Fig. \ref{fig:K_Psi_tc}. 
\begin{figure}[h!]
    \centering
    \includegraphics[scale = 0.3]{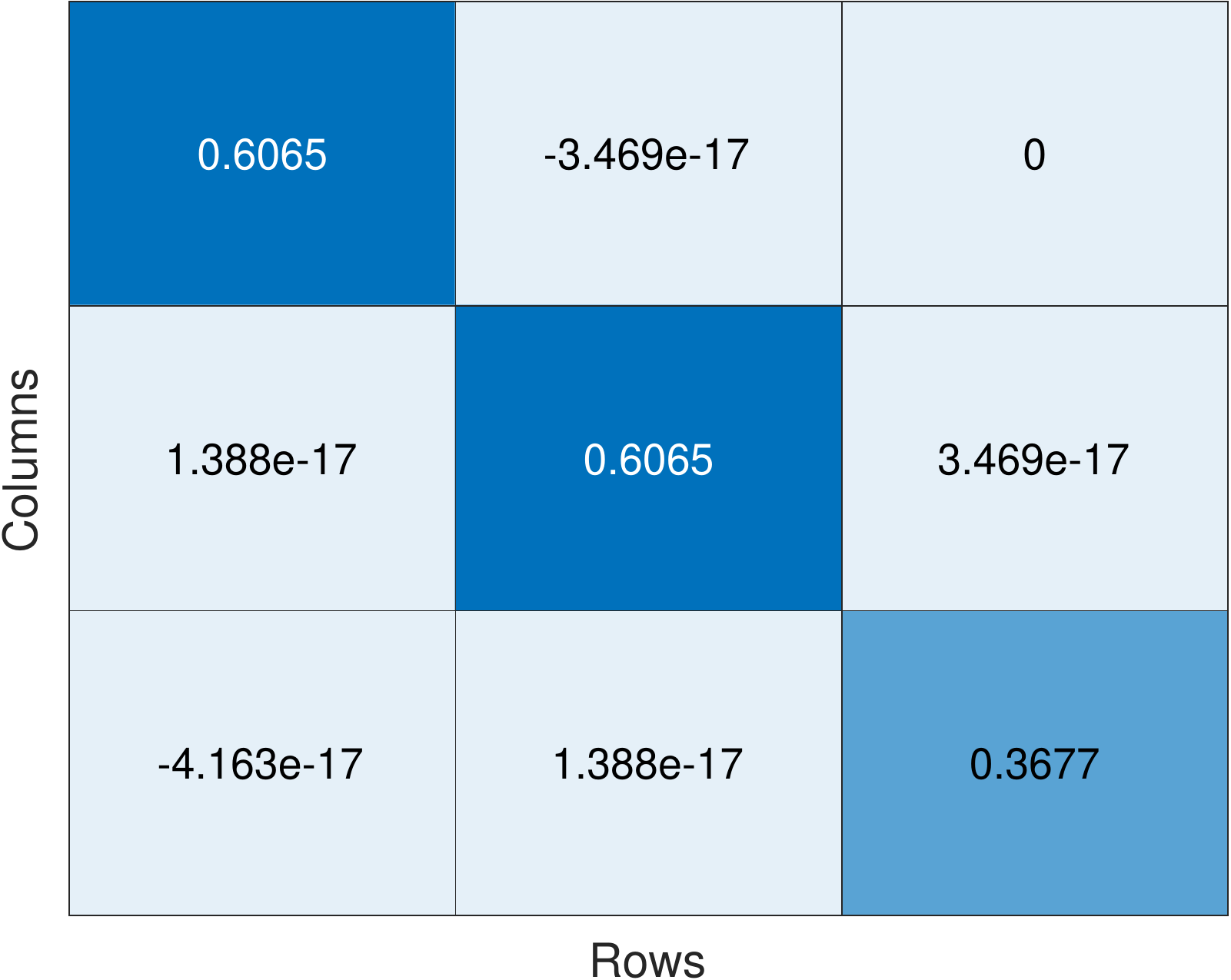}
    \caption{The matrix representation of the Koopman operator ${\cal K}_{\Psi}$ with the choice of observables $\Psi$ defined in Eq. \eqref{eq:obsvs_theta}.}
    \label{fig:K_Psi_tc}
\end{figure}
It is clear from Figs. \ref{fig:K_Theta_tc} and \ref{fig:K_Psi_tc} the Koopman operators, ${\cal K}_{\Theta}$ and ${\cal K}_{\Psi}$ are same and validates the result of Theorem \ref{thm:observable_functions_TC} on topological conjugate systems under the knowledge of the homeomorphic function.

%
%
\section{Conclusion}
\label{sec:conclusion}
Data-driven operator theoretic methods are developed for global phase analysis and learning of nonlinear dynamical systems. When new spatial time-series data arrives, we proposed and proved conditions on the residual error function under which the Koopman operator must be modified. We also proposed and proved  conditions for discovering new invariant subspaces in the state space of a dynamical system exclusively from the time-series data, based on the spectral properties of the Koopman operator. In contrast to phase space exploration results, when the local behavior in each invariant subspace is known, we developed a phase space stitching result by combining the local evolutions to predict global system dynamics. 
Next, we considered the scenario where full experimental measurements are not available, e.g., some portion of the phase space is difficult to measure or access. However, under the knowledge of symmetry relation between the invariant subspaces, that is the system is equivariant, we show the global model estimation is possible using the time-series data from only one of the invariant subspace. 
Finally, we show that global-to-local decompositions of the Koopman operator for a dynamical system can also exploit topological conjugacy, to simultaneously find model decompositions for topologically conjugate dynamical systems.  When applying our local-to-global estimation algorithms, we found we were able to reconstruct an approximation to the global Koopman operator for two example systems: a genetic toggle switch and nonlinear system with quadratic and bilinear terms. Furthermore, on topologically conjugate systems, we demonstrate that the phase space evolution of one system can be identified with the time-series data from the other system using the topological conjugacy. Future work will center on the development of software and recommender algorithms for active learning, based on the theory from this paper.

\section{Acknowledgments}
This work was supported partially by Defense Advanced Research Projects Agency (DARPA) Grants. DEAC0-576RL01830, HR0011-12-C-0065, and FA8750-19-2-0502, Institute of Collaborative Biotechnologies Grant, Army Research Office Young Investigator Program Grant W911NF-20-1-0165. Any opinions, findings and conclusions or recommendations expressed in this material are those of the author(s) and do not necessarily reflect the views of the DARPA, the Army Research Office, the Institute of Collaborative Biotechnologies, the Department of Defense, or the United States Government. 

\begin{thebibliography}{10}
\expandafter\ifx\csname url\endcsname\relax
  \def\url#1{\texttt{#1}}\fi
\expandafter\ifx\csname urlprefix\endcsname\relax\def\urlprefix{URL }\fi
\expandafter\ifx\csname href\endcsname\relax
  \def\href#1#2{#2} \def\path#1{#1}\fi

\bibitem{mezic2005spectral}
I.~Mezi{\'c}, Spectral properties of dynamical systems, model reduction and
  decompositions, Nonlinear Dynamics 41~(1-3) (2005) 309--325.

\bibitem{lasota2013chaos}
A.~Lasota, M.~C. Mackey, Chaos, fractals, and noise: stochastic aspects of
  dynamics, Vol.~97, Springer Science \& Business Media, 2013.

\bibitem{rowley2009spectral}
C.~W. Rowley, I.~Mezi{\'c}, S.~Bagheri, P.~Schlatter, D.~S. Henningson,
  Spectral analysis of nonlinear flows, Journal of fluid mechanics 641 (2009)
  115--127.

\bibitem{vaidya_lyapunov_measure}
U.~Vaidya, P.~G. Mehta, Lyapunov measure for almost everywhere stability, IEEE
  Transactions on Automatic Control 53~(1) (2008) 307--323.

\bibitem{mezic2013analysis}
I.~Mezi{\'c}, {Analysis of fluid flows via spectral properties of the Koopman
  operator}, Annual Review of Fluid Mechanics 45 (2013) 357--378.

\bibitem{yeung2017learning}
E.~Yeung, S.~Kundu, N.~Hodas, Learning deep neural network representations for
  koopman operators of nonlinear dynamical systems, in: 2019 American Control
  Conference (ACC), IEEE, 2019, pp. 4832--4839.

\bibitem{koopman1931hamiltonian}
B.~O. Koopman, {Hamiltonian systems and transformation in Hilbert space},
  Proceedings of the National Academy of Sciences 17~(5) (1931) 315--318.

\bibitem{schmid2010dynamic}
P.~J. Schmid, Dynamic mode decomposition of numerical and experimental data,
  Journal of fluid mechanics 656 (2010) 5--28.

\bibitem{proctor2016dynamic}
J.~L. Proctor, S.~L. Brunton, J.~N. Kutz, Dynamic mode decomposition with
  control, SIAM Journal on Applied Dynamical Systems 15~(1) (2016) 142--161.

\bibitem{brunton2016koopman}
S.~L. Brunton, B.~W. Brunton, J.~L. Proctor, J.~N. Kutz, {Koopman invariant
  subspaces and finite linear representations of nonlinear dynamical systems
  for control}, PloS One 11~(2) (2016) e0150171.

\bibitem{huang2018feedback}
B.~Huang, X.~Ma, U.~Vaidya, Feedback stabilization using koopman operator, in:
  2018 IEEE Conference on Decision and Control (CDC), IEEE, 2018, pp.
  6434--6439.

\bibitem{williams2015data}
M.~O. Williams, I.~G. Kevrekidis, C.~W. Rowley, {A data--driven approximation
  of the koopman operator: Extending dynamic mode decomposition}, Journal of
  Nonlinear Science 25~(6) (2015) 1307--1346.

\bibitem{mezic2020spectrum}
I.~Mezi{\'c}, Spectrum of the koopman operator, spectral expansions in
  functional spaces, and state-space geometry, Journal of Nonlinear Science
  30~(5) (2020) 2091--2145.

\bibitem{vaidya2007observability}
U.~Vaidya, Observability gramian for nonlinear systems, in: Decision and
  Control, 2007 46th IEEE Conference on, IEEE, 2007, pp. 3357--3362.

\bibitem{surana2016linear}
A.~Surana, A.~Banaszuk, {Linear observer synthesis for nonlinear systems using
  Koopman operator framework}, IFAC-PapersOnLine 49~(18) (2016) 716--723.

\bibitem{yeung2018koopman}
E.~Yeung, Z.~Liu, N.~O. Hodas, A koopman operator approach for computing and
  balancing gramians for discrete time nonlinear systems, in: 2018 Annual
  American Control Conference (ACC), IEEE, 2018, pp. 337--344.

\bibitem{arbabi2017ergodic}
H.~Arbabi, I.~Mezic, {Ergodic theory, dynamic mode decomposition, and
  computation of spectral properties of the Koopman operator}, SIAM Journal on
  Applied Dynamical Systems 16~(4) (2017) 2096--2126.

\bibitem{arbabi2017study}
H.~Arbabi, I.~Mezi{\'c}, Study of dynamics in post-transient flows using
  koopman mode decomposition, Physical Review Fluids 2~(12) (2017) 124402.

\bibitem{sinha_IT_CDC_2015}
S.~Sinha, U.~Vaidya, Formalism for information transfer in dynamical network,
  in: 2015 54th IEEE Conference on Decision and Control (CDC), IEEE, 2015, pp.
  5731--5736.

\bibitem{sinha_IT_CDC_2016}
S.~Sinha, U.~Vaidya, Causality preserving information transfer measure for
  control dynamical system, in: 2016 IEEE 55th Conference on Decision and
  Control (CDC), IEEE, 2016, pp. 7329--7334.

\bibitem{sinha_IT_ICC_2017}
S.~Sinha, U.~Vaidya, On information transfer in discrete dynamical systems, in:
  2017 Indian Control Conference (ICC), IEEE, 2017, pp. 303--308.

\bibitem{sinha_data_IT_journal}
S.~Sinha, U.~Vaidya, On data-driven computation of information transfer for
  causal inference in discrete-time dynamical systems, Journal of Nonlinear
  Science 30~(4) (2020) 1651--1676.

\bibitem{sinha_data_IT_stability}
S.~Sinha, P.~Sharma, U.~Vaidya, V.~Ajjarapu, On information transfer-based
  characterization of power system stability, IEEE Transactions on Power
  Systems 34~(5) (2019) 3804--3812.

\bibitem{susuki2016applied}
Y.~Susuki, I.~Mezic, F.~Raak, T.~Hikihara, {Applied Koopman operator theory for
  power systems technology}, Nonlinear Theory and Its Applications, IEICE 7~(4)
  (2016) 430--459.

\bibitem{barocio2014dynamic}
E.~Barocio, B.~C. Pal, N.~F. Thornhill, A.~R. Messina, A dynamic mode
  decomposition framework for global power system oscillation analysis, IEEE
  Transactions on Power Systems 30~(6) (2014) 2902--2912.

\bibitem{hernandez2018nonlinear}
M.~Hernandez-Ortega, A.~Messina, Nonlinear power system analysis using koopman
  mode decomposition and perturbation theory, IEEE Transactions on Power
  Systems 33~(5) (2018) 5124--5134.

\bibitem{raak2015data}
F.~Raak, Y.~Susuki, T.~Hikihara, Data-driven partitioning of power networks via
  koopman mode analysis, IEEE Transactions on Power Systems 31~(4) (2015)
  2799--2808.

\bibitem{netto2018robust}
M.~Netto, L.~Mili, A robust data-driven koopman kalman filter for power systems
  dynamic state estimation, IEEE Transactions on Power Systems 33~(6) (2018)
  7228--7237.

\bibitem{ramos2019dynamic}
J.~J. Ramos, J.~N. Kutz, Dynamic mode decomposition and sparse measurements for
  characterization and monitoring of power system disturbances, arXiv preprint
  arXiv:1906.03544.

\bibitem{nandanoori2020model}
S.~P. Nandanoori, S.~Kundu, S.~Pal, K.~Agarwal, S.~Choudhury, Model-agnostic
  algorithm for real-time attack identification in power grid using koopman
  modes, in: 2020 IEEE International Conference on Communications, Control, and
  Computing Technologies for Smart Grids (SmartGridComm), IEEE, 2020, pp. 1--6.

\bibitem{marrouch2019data}
N.~Marrouch, J.~Slawinska, D.~Giannakis, H.~L. Read, Data-driven koopman
  operator approach for computational neuroscience, Annals of Mathematics and
  Artificial Intelligence (2019) 1--19.

\bibitem{slawinska2019quantum}
J.~Slawinska, A.~Ourmazd, D.~Giannakis, A quantum mechanical approach for data
  assimilation in climate dynamics, in: International Conference on Machine
  Learning Workshop on ``Climate Change: How Can AI Help?'', 2019.

\bibitem{tu2013dynamic}
J.~H. Tu, C.~W. Rowley, D.~M. Luchtenburg, S.~L. Brunton, J.~N. Kutz, On
  dynamic mode decomposition: theory and applications, arXiv preprint
  arXiv:1312.0041.

\bibitem{williams2014kernel}
M.~O. Williams, C.~W. Rowley, I.~G. Kevrekidis, {A kernel-based approach to
  data-driven Koopman spectral analysis}, arXiv preprint arXiv:1411.2260.

\bibitem{sinha_robust_acc}
S.~Sinha, B.~Huang, U.~Vaidya, {Robust approximation of Koopman operator and
  prediction in random dynamical systems}, in: 2018 Annual American Control
  Conference (ACC), IEEE, 2018, pp. 5491--5496.

\bibitem{sinha_robust_journal}
S.~Sinha, B.~Huang, U.~Vaidya, {On robust computation of koopman operator and
  prediction in random dynamical systems}, Journal of Nonlinear Science (2019)
  1--34.

\bibitem{sinha_sparse_acc}
S.~Sinha, U.~Vaidya, E.~Yeung, {On computation of Koopman operator from sparse
  data}, in: 2019 American Control Conference (ACC), IEEE, 2019, pp.
  5519--5524.

\bibitem{boddupalli2019koopman}
N.~Boddupalli, A.~Hasnain, S.~P. Nandanoori, E.~Yeung, {Koopman operators for
  generalized persistence of excitation conditions for nonlinear systems}, in:
  2019 IEEE 58th Conference on Decision and Control (CDC), IEEE, 2019, pp.
  8106--8111.

\bibitem{sinha2020data}
S.~Sinha, S.~P. Nandanoori, E.~Yeung, Data driven online learning of power
  system dynamics, in: 2020 IEEE Power \& Energy Society General Meeting
  (PESGM), IEEE, 2020, pp. 1--5.

\bibitem{sinha2020equivariant}
S.~Sinha, S.~P. Nandanoori, E.~Yeung, Koopman operator methods for global phase
  space exploration of equivariant dynamical systems, IFAC-PapersOnLine 53~(2)
  (2020) 1150--1155.

\bibitem{bakker2019learning}
C.~Bakker, K.~E. Nowak, W.~S. Rosenthal, Learning koopman operators for systems
  with isolated critical points, in: 2019 IEEE 58th Conference on Decision and
  Control (CDC), IEEE, 2019, pp. 7733--7739.

\bibitem{li2017extended}
Q.~Li, F.~Dietrich, E.~M. Bollt, I.~G. Kevrekidis, Extended dynamic mode
  decomposition with dictionary learning: A data-driven adaptive spectral
  decomposition of the koopman operator, Chaos: An Interdisciplinary Journal of
  Nonlinear Science 27~(10) (2017) 103111.

\bibitem{huang2017data}
B.~Huang, U.~Vaidya, Data-driven approximation of transfer operators: Naturally
  structured dynamic mode decomposition, in: 2018 Annual American Control
  Conference (ACC), IEEE, 2018, pp. 5659--5664.

\bibitem{takeishi2017learning}
N.~Takeishi, Y.~Kawahara, T.~Yairi, Learning koopman invariant subspaces for
  dynamic mode decomposition, in: Advances in Neural Information Processing
  Systems, 2017, pp. 1130--1140.

\bibitem{lusch2018deep}
B.~Lusch, J.~N. Kutz, S.~L. Brunton, Deep learning for universal linear
  embeddings of nonlinear dynamics, Nature communications 9~(1) (2018) 4950.

\bibitem{hasnain2020steady}
A.~Hasnain, N.~Boddupalli, S.~Balakrishnan, E.~Yeung, Steady state programming
  of controlled nonlinear systems via deep dynamic mode decomposition, in: 2020
  American Control Conference (ACC), IEEE, 2020, pp. 4245--4251.

\bibitem{zhang2017evaluating}
H.~Zhang, S.~Dawson, C.~W. Rowley, E.~A. Deem, L.~N. Cattafesta, Evaluating the
  accuracy of the dynamic mode decomposition, arXiv preprint arXiv:1710.00745.

\bibitem{johnson2018class}
C.~A. Johnson, E.~Yeung, {A class of logistic functions for approximating
  state-inclusive Koopman operators}, in: 2018 Annual American Control
  Conference (ACC), IEEE, 2018, pp. 4803--4810.

\bibitem{nandanoori2020data}
S.~P. Nandanoori, S.~Sinha, E.~Yeung, Data-driven operator theoretic methods
  for global phase space learning, in: 2020 American Control Conference (ACC),
  IEEE, 2020, pp. 4551--4557.

\bibitem{budivsic2012applied}
M.~Budi{\v{s}}i{\'c}, R.~Mohr, I.~Mezi{\'c}, {Applied Koopmanism}, Chaos: An
  Interdisciplinary Journal of Nonlinear Science 22~(4) (2012) 047510.

\bibitem{susuki2015prony}
Y.~Susuki, I.~Mezi{\'c}, {A Prony approximation of Koopman mode decomposition},
  in: 2015 54th IEEE Conference on Decision and Control (CDC), IEEE, 2015, pp.
  7022--7027.

\bibitem{bagheri2013koopman}
S.~Bagheri, Koopman-mode decomposition of the cylinder wake, Journal of Fluid
  Mechanics 726 (2013) 596--623.

\bibitem{sharma2016correspondence}
A.~S. Sharma, I.~Mezi{\'c}, B.~J. McKeon, Correspondence between koopman mode
  decomposition, resolvent mode decomposition, and invariant solutions of the
  navier-stokes equations, Physical Review Fluids 1~(3) (2016) 032402.

\bibitem{petersen1989ergodic}
K.~E. Petersen, Ergodic theory, Vol.~2, Cambridge University Press, 1989.

\bibitem{horn2012matrix}
R.~A. Horn, C.~R. Johnson, Matrix analysis, Cambridge university press, 2012.

\bibitem{gardner2000construction}
T.~S. Gardner, C.~R. Cantor, J.~J. Collins, {Construction of a genetic toggle
  switch in Escherichia coli}, Nature 403~(6767) (2000) 339.

\bibitem{tian2006stochastic}
T.~Tian, K.~Burrage, Stochastic models for regulatory networks of the genetic
  toggle switch, Proceedings of the national Academy of Sciences 103~(22)
  (2006) 8372--8377.

\bibitem{munsky2010guidelines}
B.~Munsky, M.~Khammash, Guidelines for the identification of a stochastic model
  for the genetic toggle switch, IET Systems Biology 4 (2010) 356--366.

\bibitem{yeung2021data}
E.~Yeung, J.~Kim, Y.~Yuan, J.~Goncalves, R.~M. Murray, Data-driven network
  models for genetic circuits from time-series data with incomplete
  measurements, bioRxiv.

\bibitem{gyorgy2016quantifying}
A.~Gyorgy, R.~M. Murray, {Quantifying resource competition and its effects in
  the TX-TL system}, in: Decision and Control (CDC), 2016 IEEE 55th Conference
  on, IEEE, 2016, pp. 3363--3368.

\end{thebibliography}

\end{document}